\documentclass[10pt]{article}

\usepackage{color}\usepackage{graphicx}\usepackage{mathrsfs,amsfonts,latexsym,amssymb,amsthm,amsmath}
\usepackage{anysize}\marginsize{2cm}{2cm}{1cm}{2.5cm}
\usepackage[font=footnotesize]{caption}

\newtheorem{Theorem}{Theorem}\newtheorem{Lemma}{Lemma}\newtheorem{Proposition}{Proposition}
\theoremstyle{definition}\newtheorem{Definition}{Definition}\newtheorem{Remark}{Remark}\newtheorem{Example}{Example}\newtheorem{Case}{Case}

\begin{document}

\title{Spectral behavior of preconditioned non-Hermitian multilevel block Toeplitz matrices with matrix-valued symbol}

\author{Marco Donatelli, Carlo Garoni, Mariarosa Mazza, Stefano Serra-Capizzano, Debora Sesana\footnote{Department of Science and High Technology, University of Insubria, Via Valleggio 11, 22100 Como, Italy \newline \emph{Email addresses}: \texttt{marco.donatelli@uninsubria.it} (Marco Donatelli), \texttt{carlo.garoni@uninsubria.it} (Carlo Garoni), \newline \texttt{mariarosa.mazza@uninsubria.it} (Mariarosa Mazza), \texttt{stefano.serrac@uninsubria.it} (Stefano Serra-Capizzano), \newline \texttt{debora.sesana@uninsubria.it} (Debora Sesana)}}

\maketitle

\begin{abstract}
This note is devoted to preconditioning strategies for non-Hermitian multilevel block Toeplitz linear systems associated with a multivariate Lebesgue integrable matrix-valued symbol.
In particular, we consider special preconditioned matrices, where the preconditioner has a band multilevel block Toeplitz structure, and we complement known results on the localization of the spectrum with global distribution results for the eigenvalues of the preconditioned matrices. In this respect, our main result is as follows. Let $I_k:=(-\pi,\pi)^k$, let $\mathcal M_s$ be the linear space of complex $s\times s$ matrices, and let $f,g:I_k\to\mathcal M_s$ be functions whose components $f_{ij},\,g_{ij}:I_k\to\mathbb C,\ i,j=1,\ldots,s,$ belong to $L^\infty$. Consider the matrices $T_n^{-1}(g)T_n(f)$, where $n:=(n_1,\ldots,n_k)$ varies in $\mathbb N^k$ and $T_n(f),T_n(g)$ are the multilevel block Toeplitz matrices of size $n_1\cdots n_ks$ generated by $f,g$. Then $\{T_n^{-1}(g)T_n(f)\}_{n\in\mathbb N^k}\sim_\lambda g^{-1}f$, i.e. the family of matrices $\{T_n^{-1}(g)T_n(f)\}_{n\in\mathbb N^k}$ has a global (asymptotic) spectral distribution described by the function $g^{-1}f$, provided $g$ possesses certain properties (which ensure in particular the invertibility of $T_n^{-1}(g)$ for all $n$) and the following topological conditions are met: the essential range of $g^{-1}f$, defined as the union of the essential ranges of the eigenvalue functions $\lambda_j(g^{-1}f),\ j=1,\ldots,s$, does not disconnect the complex plane and has empty interior. This result generalizes the one obtained by Donatelli, Neytcheva, Serra-Capizzano in a previous work, concerning the non-preconditioned case $g=1$. The last part of this note is devoted to numerical experiments, which confirm the theoretical analysis and suggest the choice of optimal GMRES preconditioning techniques to be used for the considered linear systems.
\end{abstract}



\section{Introduction}\label{sec1}

Let $\mathcal M_s$ be the linear space of the complex $s\times s$ matrices. We say that a function $f:G\to\mathcal M_s$, defined on some measurable set $G\subseteq\mathbb R^k$, is in $L^p(G)$/measurable/continuous, if its components $f_{ij}:G\to\mathbb C,\ i,j=1,\ldots,s,$ are in $L^p(G)$/measurable/continuous. Moreover, we denote by $I_k$ the $k$-cube $(-\pi,\pi)^k$ and, for $1\le p\le\infty$, we define $L^p(k,s)$ as the linear space of $k$-variate functions $f:I_k\to\mathcal M_s$ belonging to $L^p(I_k)$.

Now fix $f\in L^1(k,s).$ The Fourier coefficients of $f$ are defined as
\begin{equation}\label{fhat}
\widehat f_j:=\frac1{(2\pi)^k}\int_{I_k}f(x){\rm e}^{-{\mathbf i}\left\langle j,x\right\rangle}dx\in\mathcal M_s, \quad j\in\mathbb Z^k,
\end{equation}
where $\mathbf i^2=-1$, $\left\langle j,x\right\rangle=\sum_{t=1}^kj_tx_t$, and the integrals in \eqref{fhat} are done componentwise. Starting from the coefficients $\widehat f_j$, we can construct the family of $k$-level (block) Toeplitz matrices generated by $f$. More in detail, if $n:=(n_1,\ldots,n_k)$ is a multi-index in $\mathbb N^k$, let $\widehat n:=\prod_{i=1}^kn_i$. Then the $n$-th Toeplitz matrix associated with $f$ is the matrix of order $s\widehat n$ given by
\begin{equation}\label{mbTm}
T_n(f)=\sum_{|j_1|<n_1}\cdots\sum_{|j_k|<n_k}\left[J_{n_1}^{(j_1)}\otimes\cdots\otimes J_{n_k}^{(j_k)}\right]\otimes\widehat f_{(j_1,\ldots,j_k)},
\end{equation}
where $\otimes$ denotes the (Kronecker) tensor product, while $J_t^{(l)}$ is the matrix of order $t$ whose $(i,j)$ entry equals $1$ if $i-j=l$ and equals zero otherwise: the reader is referred to \cite{Tillinota} for more details on multilevel block Toeplitz matrices. $\{T_n(f)\}_{n\in\mathbb N^k}$ is called the family of Toeplitz matrices generated by $f$, which in turn is called the symbol (or the generating function) of $\{T_n(f)\}_{n\in\mathbb N^k}$.

The problem considered in this paper is the numerical solution of a linear system with coefficient matrix $T_n(f)$, where $f\in L^1(k,s)$ and $n$ is a multi-index in $\mathbb N^k$ with large components $n_1,\ldots,n_k$.
Such type of linear systems arise in important applications such as Markov chains \cite{markov1,markov2} (with $k=1$ and $s>1$), in the reconstruction of signals with missing data \cite{symbol-image} (with $k=1$ and $s=2$), in the inpainting problem \cite{inpaint} (with $k=2$ and $s=2$), and of course in the numerical approximation of constant coefficient $r\times r$ systems of PDEs over $d$-dimensional domains \cite{AxB} (with $k=d$ and $s=r$).
We are interested in preconditioning $T_n(f)$ by $T_n(g)$, where $g$ is a trigonometric polynomial so that $T_n(g)$ is banded and the related linear systems are easily solvable: this is called band Toeplitz preconditioning.
In connection with Krylov methods, other popular preconditioners can be chosen in appropriate algebras of matrices with fast transforms: we refer to circulants, trigonometric and Hartley algebras, wavelet algebras, etc.
Unfortunately, general results in \cite{tyrty1,tyrty2,negablo,nsv} tell us that the performances of such preconditioners deteriorate when $k$ increases and optimal methods can be obtained only for $k=1$.
For optimal methods we mean methods such that the complexity of solving the given linear system is proportional to the cost of matrix-vector multiplication, see \cite{AxN} for a precise notion in the context of iterative methods.
Concerning the band Toeplitz preconditioning, we emphasize that the technique has been explored for $k\ge 1$ and $s=1$ in \cite{cband1,DFS,Sbit,cband2,Sopt1}, even in the (asymptotically) ill-conditioned case characterized by zeros of the symbol, but with a specific focus on the positive definite case. Specific attempts in the non-Hermitian case can be found in \cite{CChi,huckle-nonh}. Further results concerning genuine block Toeplitz structures with $s>1$ are considered in \cite{Smarko1,Smarko2}, but again for Hermitian positive definite matrix-valued symbols.
In this note, the attention is concentrated on the non-Hermitian case with general $k,s\ge 1$ and, more specifically, we are interested in the following three items.
\begin{itemize}
	\item Localization results for all the eigenvalues of $T_n^{-1}(g)T_n(f)$: the results can be found in \cite{huckle-nonh} for $f,g\in L^1(k,s)$ in the case where $g$ is a sectorial function, which implies in particular that each $T_n(g)$ is invertible (see Definition \ref{ws} for the definition of sectorial function);
	\item Spectral distribution results for the family of matrices (matrix-family) $\{T_n^{-1}(g)T_n(f)\}_{n\in\mathbb N^k}$: this is our original contribution and it represents a generalization of \cite[Theorem 1.2]{maya}, where $g$ is assumed to be identically equal to $1$. We recall that general results on the distribution of the singular values of families such as $\{T_n^{-1}(g)T_n(f)\}_{n\in\mathbb N^k}$ are known in the wide context of Generalized Locally Toeplitz sequences, see \cite{glt-vs-fourier} and references therein;
	\item A wide set of numerical experiments concerning the eigenvalue localization and the clustering properties of the matrix $T_n^{-1}(g)T_n(f)$, and regarding the effectiveness of the preconditioned GMRES, with preconditioning strategies chosen according to the theoretical indications given in the first two items.
\end{itemize}

The paper is organized as follows. Section \ref{loc} introduces the notions of weakly sectorial and sectorial function and deals with the localization results for the eigenvalues of $T_n^{-1}(g)T_n(f)$ in the case where $g$ is sectorial. Section \ref{spec} contains our main result (Theorem \ref{szego-nonherm-prec}) concerning the asymptotic spectral distribution of the matrix-family $\{T_n^{-1}(g)T_n(f)\}_{n\in\mathbb N^k}$. Finally, Section \ref{num} is devoted to numerical experiments which take advantage of the theoretical results in order to devise suitable GMRES preconditioners for linear systems with coefficient matrix $T_n(f),\ f\in L^\infty(k,s)$.

Before starting with our study, let us fix the multi-index notation that will be extensively used throughout this paper. A multi-index $m\in\mathbb Z^k$, also called a $k$-index, is simply a vector in $\mathbb Z^k$ and its components are denoted by $m_1,\ldots,m_k$. Standard operations defined for vectors in $\mathbb C^k$, such as addition, subtraction and scalar multiplication, are also defined for $k$-indices. We will use the letter $e$ for the vector of all ones, whose size will be clear from the context. If $m$ is a $k$-index, we set $\widehat m:=\prod_{i=1}^km_i$ and we write $m\to\infty$ to indicate that $\min_im_i\to\infty$. Inequalities involving multi-indices are always understood in the componentwise sense. For instance, given $h,m\in\mathbb Z^k$, the inequality $h\le m$ means that $h_l\le m_l$ for all $l=1,\ldots,k$. If $h,m\in\mathbb Z^k$ and $h\le m$, the multi-index range $h,\ldots,m$ is the set $\{j\in\mathbb Z^k:h\le j\le m\}$. We assume for the multi-index range $h,\ldots,m$ the standard lexicographic ordering:
\begin{equation}\label{ordering}
\left[\ \ldots\ \left[\ \left[\ (j_1,\ldots,j_k)\ \right]_{j_k=h_k,\ldots,m_k}\ \right]_{j_{k-1}=h_{k-1},\ldots,m_{k-1}}\ \ldots\ \right]_{j_1=h_1,\ldots,m_1}.
\end{equation}
For instance, if $k=2$ then the ordering is {\footnotesize$$(h_1,h_2),\:(h_1,h_2+1),\:\ldots,\:(h_1,m_2),\:(h_1+1,h_2),\:(h_1+1,h_2+1),\:\ldots,\:(h_1+1,m_2),\:\ldots\:\ldots,\:(m_1,h_2),\:(m_1,h_2+1),\:\ldots,\:(m_1,m_2).$$}When a multi-index $j$ varies over a multi-index range $h,\ldots,m$ (this may be written as $j=h,\ldots,m$),
it is always understood that $j$ varies from $h$ to $m$ according to the lexicographic ordering \eqref{ordering}. For instance, if $m\in\mathbb N^k$ and if $y=\left[y_i\right]_{i=e}^m$, then $y$ is a vector of size $m_1\cdots m_k$ whose components $y_i,\ i=e,\ldots,m,$ are ordered in accordance with the ordering \eqref{ordering} for the multi-index range $e,\ldots,m$. Similarly, if $Y=\left[y_{ij}\right]_{i,j=e}^m$, then $Y$ is a matrix of size $m_1\cdots m_k$ whose components are indexed by two multi-indices $i,j$, both varying over the multi-index range $e,\ldots,m$ in accordance with \eqref{ordering}.
To conclude, we point out that the multilevel block Toeplitz matrix $T_n(f)$ displayed in \eqref{mbTm} can be expressed in multi-index notation as
\begin{equation}\label{mbTm_multi-index}
T_n(f)=[\widehat f_{i-j}]_{i,j=e}^n.
\end{equation}

\section{Localization results for the eigenvalues of $T_n^{-1}(g)T_n(f)$}\label{loc}

Let us recall in this section important localization results taken from \cite{jia,huckle-nonh}. We first introduce the notions of essential range ${\cal ER}(f)$ and essential numerical range ${\cal ENR}(f)$ of a matrix-valued function $f$. In the following, for any $X\subseteq\mathbb C$, ${\rm Coh}[X]$ is the convex hull of $X$ and $d(X,z)$ is the (Euclidean) distance of $X$ from the point $z\in\mathbb C$. We denote by $\|\cdot\|$ the spectral (Euclidean) norm of both vectors and matrices. If $r\in\mathbb C$ and $\epsilon>0$, $D(r,\epsilon)$ is the disk in the complex plane centered at $r$ with radius $\epsilon$. Recall that, if $h:G\to\mathbb C$ is a complex-valued measurable function, defined on some measurable set $G\subseteq\mathbb R^k$, the essential range of $h$, $\mathcal{ER}(h)$, is defined as the set of points $r\in\mathbb C$ such that, for every $\epsilon>0$, the measure of $h^{-1}(D(r,\epsilon)):=\{t\in G:h(t)\in D(r,\epsilon)\}$ is positive. In symbols,
$$\mathcal{ER}(h):=\{r\in\mathbb C:\forall\epsilon>0,\ m_k\{t\in G:h(t)\in D(r,\epsilon)\}>0\},$$
where $m_k$ is the Lebesgue measure in $\mathbb R^k$. Note that $\mathcal{ER}(h)$ is always closed (the complement is open). Moreover, it can be shown that $h(t)\in\mathcal{ER}(h)$ for almost every $t\in G$, i.e., $h\in\mathcal{ER}(h)$ a.e.

\begin{Definition}\label{er-enr}
Given a measurable matrix-valued function $f:G\to\mathcal M_s$, defined on some measurable set $G\subseteq\mathbb R^k$,
\begin{itemize}
\item the {\em essential range} of $f$, $\mathcal{ER}(f)$, is the union of the essential ranges of the eigenvalue functions $\lambda_j(f):G\to\mathbb C$, $j=1,\ldots,s$, that is $\mathcal{ER}(f):=\bigcup_{j=1}^s\mathcal{ER}(\lambda_j(f))$;
\item the {\em essential numerical range} of $f$, $\mathcal{ENR}(f)$, is the set of points $r\in\mathbb C$ such that, for every $\epsilon>0$, the measure of $\{t\in G: \exists v\in\mathbb C^s\ \textup{with}\ \|v\|=1\ \textup{such that}\ v^*f(t)v\in D(r,\epsilon)\}$ is positive. In symbols,
$$\mathcal{ENR}(f):=\{r\in\mathbb C:\forall\epsilon>0,\ m_k\{t\in G:\exists v\in\mathbb C^s\ \textup{with}\ \|v\|=1\ \textup{such that}\ v^*f(t)v\in D(r,\epsilon)\}>0\}.$$
\end{itemize}
\end{Definition}
Note that ${\cal ER}(f)$ is closed, being the union of a finite number of closed sets. $\mathcal{ENR}(f)$ is also closed, because its complement is open. Moreover, it can be proved that, for a.e. $t\in G$, the following property holds: $v^*f(t)v\in\mathcal{ENR}(f)$ for all $v\in\mathbb C^s$ with $\|v\|=1$. In other words, $v^*fv\in\mathcal{ENR}(f)$ for all $v\in\mathbb C^s$ with $\|v\|=1$, a.e. In addition, it can be shown that $\mathcal{ENR}(f)\supseteq\mathcal{ER}(f)$. In the case $s=1$, we have $\mathcal{ENR}(f)=\mathcal{ER}(f)$.

Now we turn to the definition of sectorial function. Given a straight line $z$ in the complex plane, let $H_1$ and $H_2$ be the two open half-planes such that $\mathbb C$ is the disjoint union $H_1\bigcup z\bigcup H_2$; we call $H_1$ and $H_2$ the open half-planes determined by $z$. Moreover, we denote by $\omega(z)\in\mathbb C$ the rotation number (of modulus 1) such that $\omega(z)\cdot z=\{w\in\mathbb C:{\rm Re}(w)=d(z,0)\}$.
\begin{Definition}\label{ws}
A function $f\in L^1(k,s)$ is {\em weakly sectorial} if there exists a straight line $z$ in the complex plane with the following property: one of the two open half-planes determined by $z$, say $H_1$, is such that $\mathcal{ENR}(f)\bigcap H_1=\emptyset$ and $0\in H_1\bigcup z$.
Whenever $f\in L^1(k,s)$ is weakly sectorial, every straight line $z$ with the previous property is called a {\em separating line} for $\mathcal{ENR}(f)$. A function $f\in L^1(k,s)$ is {\em sectorial} if it is weakly sectorial and there exists a separating line $z$ such that the minimal eigenvalue of $\omega(z)f(x)+\overline{\omega(z)}f^*(x)$ is not a.e. equal to $d(z,0)$.
\end{Definition}

The following Lemma provides simple conditions that ensure a given function $f\in L^1(k,s)$ to be weakly sectorial or sectorial. We do not prove this Lemma, because it is beyond the purpose of this paper. We limit to say that the proof can be obtained using the following topological properties of convex sets: the separability properties provided by the geometric forms of the Hanh-Banach theorem, see \cite[Theorems 1.6 and 1.7]{Brezis}; the result stating that, for any convex set $C$, the closure $\overline C$ and the interior ${\rm Int}(C)$ are convex, and $\overline{{\rm Int}(C)}=\overline C$ whenever ${\rm Int}(C)$ is nonempty, see e.g. \cite[Exercise 1.7]{Brezis}.

\begin{Lemma}\label{lemma_geometrico}
Let $f\in L^1(k,s)$.
\begin{itemize}
	\item $f$ is weakly sectorial if and only if $0\notin{\rm Int}({\rm Coh}[\mathcal{ENR}(f)])$.
	\item If $0\notin\overline{{\rm Coh}[\mathcal{ENR}(f)]}$ then $f$ is sectorial. Equivalently, if $d({\rm Coh}[\mathcal{ENR}(f)],0)>0$ then $f$ is sectorial.
\end{itemize}
\end{Lemma}

\begin{Theorem}\label{locsv4}
{\rm\cite{jia}} Let $f\in L^1(k,s)$ and let $d:=d({\rm Coh}[{\cal ENR}(f)],0)$.
\begin{itemize}
	\item Suppose $f$ is weakly sectorial. Then $\sup_{z\in\mathcal S}d(z,0)=\max_{z\in\mathcal S}d(z,0)=d$, where $\mathcal S$ is the set of all separating lines for $\mathcal{ENR}(f)$. Moreover, $\sigma\ge d$ for all singular values $\sigma$ of $T_n(f)$ and for all $n\in\mathbb N^k$.
	\item Suppose $f$ is sectorial and let $z$ be a separating line for $\mathcal{ENR}(f)$ such that the minimal eigenvalue of $\omega(z)f(x)+\overline{\omega(z)}f^*(x)$ is not a.e. equal to $d(z,0)$. Then $\sigma>d(z,0)$ for all singular values $\sigma$ of $T_n(f)$ and for all $n\in\mathbb N^k$.
\end{itemize}
In particular, if $f$ is sectorial then all the matrices $T_n(f),\ n\in\mathbb N^k,$ are invertible.
\end{Theorem}

We remark that, if $f\in L^1(k,s)$ and if $\widetilde f(x)$ is similar to $f(x)$ via a constant transformation $C$ (independent of $x$), that is $f(x)=C\widetilde f(x)C^{-1}$ a.e., then $\widetilde f\in L^1(k,s)$ and $T_n(\widetilde f)=(I_{\widehat n}\otimes C)^{-1}T_n(f)(I_{\widehat n}\otimes C)$ for all $n\in\mathbb N^k$ ($I_{\widehat n}$ is the identity matrix of order $\widehat n$). This result follows from the definitions of $T_n(\widetilde f),T_n(f)$, see \eqref{mbTm}, and from the properties of the tensor product of matrices.

\begin{Theorem} \label{loc-eigp-nh}{\rm \cite{huckle-nonh}}
Suppose $f,g\in L^1(k,s)$ with $g$ sectorial, and let $R(f,g):=\{\lambda\in {\mathbb C}: f-\lambda g\mbox{\ is sectorial}\}$.
Then, for any $n$, the eigenvalues of $T_n^{-1}(g)T_n(f)$ belong to $[R(f,g)]^c$, i.e. to the complementary set of $R(f,g)$.
In addition, if $\widetilde f(x)$ is similar to $f(x)$ via a constant transformation and if $\widetilde g$ is similar to $g$ via the same constant transformation, then $T_n^{-1}(g)T_n(f)$ is similar to $T_n^{-1}(\widetilde g)T_n(\widetilde f)$ by the above discussion and therefore, for any $n$, the eigenvalues of $T_n^{-1}(g)T_n(f)$ belong to $[R(\widetilde f,\widetilde g)]^c$ as well.
As a consequence, if $\cal F$ denotes the set of all pairs $(\widetilde f,\widetilde g)$ satisfying the previous assumptions, then, for any $n$, the eigenvalues of $T_n^{-1}(g)T_n(f)$ belong to $\bigcap_{(\widetilde f,\widetilde g)\in \cal F} [R(\widetilde f,\widetilde g)]^c$.
\end{Theorem}

Except for the study of preconditioning strategies associated with the (preconditioned) normal equations (see e.g. \cite{CN}), whose related numerical results are rarely satisfactory in the ill-conditioned case, there are no specialized preconditioning techniques for non-Hermitian multilevel block Toeplitz matrices. Theorem \ref{loc-eigp-nh} (straightforward block extension of a theorem taken from \cite{jia}) is, to our knowledge, the first tool for devising spectrally equivalent preconditioners in the non-Hermitian multilevel block case.
We notice that in the Hermitian case there exists a wide choice of different versions of Theorem \ref{loc-eigp-nh} (see e.g. {\rm \cite{Smarko2}}).
This is the first version for the non-Hermitian block case that could be used in connection with the preconditioning.
\begin{Example}\label{Esempio}
For the sake of simplicity, set $s=2$, $j\ge e$ (we recall that $e$ is the vector of all ones), and consider, for $k=1,2$ and $x\in I_k$, the following matrix-valued functions:
\begin{equation*}
\begin{split}
f(x) & = Q(x) \left(\begin{array}{cc} \left(1-{\rm e}^{{\mathbf i}\left\langle  j,x\right\rangle}\right)\varphi_1(x) & 0 \\
0 &  \varphi_2(x)\end{array}\right) Q^{-1}(x),\\
g(x) & = Q(x) \left(\begin{array}{cc} 1-{\rm e}^{{\mathbf i}\left\langle  j,x\right\rangle} & 0 \\
0 &  1\end{array}\right) Q^{-1}(x),
\end{split}
\end{equation*}
where $\varphi_i(x),\ i=1,2,$ are real-valued and $\inf\varphi_i=r_i>0,\ \sup\varphi_i=R_i<\infty$. In this case, if $\varphi_i(x),\ i=1,2,$ and $Q(x)$ satisfy certain properties, the set $[R(f,g)]^c$ is bounded away from zero and infinity and its intersection with real line is an interval of the form $(r,R)$, with $r,R>0$. Therefore, $T_n(g)$ is an optimal preconditioner (see \cite{SS}) for $T_n(f)$.
We can use this (optimal) preconditioner in connection with classical iterative solvers 
like the Gauss-Seidel method or with methods like the GMRES \cite{SS}, whose convergence speed is strongly dependent on the localization and distribution of the eigenvalues \cite{NRT}. Note that, when the matrix $Q$ does not depend on $x$,
$$T_n(g)=(I_{\widehat n}\otimes Q)T_n(B)(I_{\widehat n}\otimes Q^{-1}),\quad\mbox{with}\quad
B(x):=\left(\begin{array}{cc} 1-{\rm e}^{{\mathbf i}\left\langle  j,x\right\rangle} & 0 \\
0 &  1\end{array}\right).$$
In this case, the solution of a linear system associated with $T_n(g)$ can be reduced to the solution of a linear system associated with $T_n(B)$ and, since $T_n(B)$ is a band lower triangular matrix, the corresponding banded linear systems can be optimally solved (both in the univariate and multivariate settings) by using a direct elimination.
If $B$ is a more general weakly sectorial trigonometric polynomial, then $T_n(B)$ is just banded and, at least in the univariate context, we recall that we can apply specialized versions of the Gaussian Elimination maintaining an optimal linear cost.
\end{Example}

Theorem \ref{loc-eigp-nh} is rather powerful but its assumptions do not seem easy to check. More precisely, a set of important problems to be considered for the practical use of Theorem \ref{loc-eigp-nh} is the following:
{\bf (a)} given $f$ regular enough, give conditions such that there exists a polynomial $g$ for which the assumptions of Theorem \ref{loc-eigp-nh} are satisfied;
{\bf (b)} let us suppose that $f$ satisfies the conditions of the first item; give a constructive way (an algorithm) for defining such a polynomial $g$. Due to the difficulty of addressing these problems, a different approach can be adopted for devising suitable preconditioners $T_n(g)$ for $T_n(f)$. More precisely, instead of looking for a precise spectral localization of $T_n^{-1}(g)T_n(f)$, we just analyze the global asymptotic behavior of the spectrum of $T_n^{-1}(g)T_n(f)$ as $n\to\infty$. As we shall see through numerical experiments in Section \ref{num}, the knowledge of the asymptotic spectral distribution of $\{T_n^{-1}(g)T_n(f)\}_{n\in\mathbb N^k}$ can indeed be useful as a guide for designing appropriate preconditioners $T_n(g)$ for $T_n(f)$.

\section{Spectral distribution results for $\{T_n^{-1}(g)T_n(f)\}_{n\in\mathbb N^k}$}\label{spec}

We begin with the definition of spectral distribution and clustering, in the sense of eigenvalues and singular values, of a sequence of matrices (matrix-sequence), and we define the area of $K$, in the case where $K$ is a compact subset of $\mathbb C$.
Then, we present the main tool, taken from \cite{maya}, for proving our main result, i.e. Theorem \ref{szego-nonherm-prec}, which provides the asymptotic spectral distribution of preconditioned multilevel block Toeplitz matrices.
Finally, Theorem \ref{szego-nonherm-prec} is stated and proved.

Before starting, let us introduce some notation. We denote by $\mathcal C_0(\mathbb C)$ and $\mathcal C_0(\mathbb R^+_0)$ the set of continuous functions with bounded support defined over $\mathbb C$ and $\mathbb R_0^+=[0,\infty)$, respectively. Given a function $F$ and given a matrix $A$ of order $m$, with eigenvalues $\lambda_j(A),\ j=1,\ldots,m,$ and singular values $\sigma_j(A),\ j=1,\ldots,m,$ we set $$\Sigma_\lambda(F,A):=\frac{1}{m}\sum_{j=1}^{m}F(\lambda_{j}(A)),\qquad \Sigma_\sigma(F,A):=\frac{1}{m}\sum_{j=1}^{m}F(\sigma_{j}(A)).$$
Moreover, ${\rm tr}(A)$ is the trace of $A$.

\begin{Definition}\label{def-distribution}
Let $f:G\to\mathcal M_s$ be a measurable function, defined on a measurable set $G\subset\mathbb R^k$ with $0<m_k(G)<\infty$.
Let $\{A_n\}$ be a matrix-sequence, with $A_n$ of size $d_n$ tending to infinity.
\begin{itemize}
	\item $\{A_n\}$ is {\em distributed as the pair $(f,G)$ in the sense of the eigenvalues,} in symbols $\{A_n\}\sim_\lambda(f,G)$, if for all $F\in\mathcal C_0(\mathbb C)$ we have
\begin{equation}\label{distribution:sv-eig}
\lim_{n\to\infty}\Sigma_\lambda(F,A_n)=\frac1{m_k(G)}\int_G\frac{\sum_{i=1}^sF(\lambda_i(f(t)))}sdt=\frac1{m_k(G)}\int_G\frac{{\rm tr}(F(f(t)))}sdt.
\end{equation}
	\item $\{A_n\}$ is {\em distributed as the pair $(f,G)$ in the sense of the singular values,} in symbols $\{A_n\}\sim_\sigma(f,G)$, if for all $F\in\mathcal C_0(\mathbb R_0^+)$ we have
\begin{equation}\label{distribution:sv-eig-bis}
\lim_{n\to\infty}\Sigma_\sigma(F,A_n)=\frac1{m_k(G)}\int_G\frac{\sum_{i=1}^sF(\sigma_i(f(t)))}sdt=\frac1{m_k(G)}\int_G\frac{{\rm tr}(F(|f(t)|))}sdt,
\end{equation}
where $|f(t)|:=(f^*(t)f(t))^{1/2}$.
\end{itemize}
If $\{A_n\}_{n\in\mathbb N^h}$ is a matrix-family (parameterized by a multi-index), with $A_n$ of size $d_n$ tending to infinity when $n\to\infty$ (i.e. when $\min_jn_j\to\infty$), we still write $\{A_n\}_{n\in\mathbb N^h}\sim_\lambda(f,G)$ to indicate that \eqref{distribution:sv-eig} is satisfied for all $F\in\mathcal C_0(\mathbb C)$, but we point out that now `$n\to\infty$' in \eqref{distribution:sv-eig} means `$\min_jn_j\to\infty$', in accordance with the multi-index notation introduced in Section \ref{sec1}. Similarly, we write $\{A_n\}_{n\in\mathbb N^h}\sim_\sigma(f,G)$ if \eqref{distribution:sv-eig-bis} is satisfied for all $F\in\mathcal C_0(\mathbb R_0^+)$, where again $n\to\infty$ means $\min_jn_j\to\infty$. We note that $\{A_n\}_{n\in\mathbb N^h}\sim_\lambda(f,G)$ (resp. $\{A_n\}_{n\in\mathbb N^h}\sim_\sigma(f,G)$) is equivalent to saying that $\{A_{n(m)}\}_m\sim_\lambda(f,G)$ (resp. $\{A_{n(m)}\}_m\sim_\sigma(f,G)$) for every matrix-sequence $\{A_{n(m)}\}_m$ extracted from $\{A_n\}_{n\in\mathbb N^h}$ and such that $\min_jn_j(m)\to\infty$ as $m\to\infty$.
\end{Definition}

For $S\subseteq\mathbb C$ and $\epsilon>0$, we denote by $D(S,\epsilon)$ the $\epsilon$-expansion of $S$, defined as $D(S,\epsilon)=\bigcup_{r\in S}D(r,\epsilon)$.

\begin{Definition}\label{def-cluster}
Let $\{A_n\}$ be a matrix-sequence, with $A_n$ of size $d_n$ tending to infinity, and let $S\subseteq\mathbb C$ be a closed subset of $\mathbb C$. We say that $\{A_n\}$ is strongly clustered at $S$ in the sense of the eigenvalues if, for every $\epsilon>0$, the number of eigenvalues of $A_n$ outside $D(S,\epsilon)$ is bounded by a constant $q_\epsilon$ independent of $n$. In other words
\begin{equation}\label{q-lambda}
q_\epsilon(n,S):=\#\{j\in\{1,\ldots,d_n\}:\lambda_j(A_n)\notin D(S,\epsilon)\}=O(1),\quad\mbox{as $n\to\infty$.}
\end{equation}
We say that $\{A_n\}$ is weakly clustered at $S$ in the sense of the eigenvalues if, for every $\epsilon>0$,
$$q_\epsilon(n,S)=o(n),\quad\mbox{as $n\to\infty$.}$$
If $\{A_n\}$ is strongly or weakly clustered at $S$ and $S$ is not connected, then its disjoint parts are called sub-clusters.

By replacing `eigenvalues' with `singular values' and $\lambda_j(A_n)$ with $\sigma_j(A_n)$ in \eqref{q-lambda}, we obtain the definitions of a matrix-sequence strongly or weakly clustered at a closed subset of $\mathbb C$, in the sense of the singular values. It is worth noting that, since the singular values are always non-negative, any matrix-sequence is strongly clustered in the sense of the singular values at a certain $S\subseteq[0,\infty)$. Similarly, any matrix-sequence formed by matrices with only real eigenvalues (e.g. by Hermitian matrices) is strongly clustered at some $S\subseteq\mathbb R$ in the sense of the eigenvalues.
\end{Definition}

\begin{Remark}\label{Golinskii}
If $\{A_n\}\sim_\lambda(f,G)$, with $\{A_n\},f,G$ as in Definition \ref{def-distribution}, then $\{A_n\}$ is weakly clustered at $\mathcal{ER}(f)$ in the sense of the eigenvalues. This result is proved in \cite[Theorem 4.2]{Golinskii-Serra}. It is clear that $\{A_n\}\sim_{\lambda}(f,G)$, with $f\equiv r$ equal to a constant function, is equivalent to saying that $\{A_n\}$ is weakly clustered at $r\in\mathbb C$ in the sense of the eigenvalues. The reader is referred to \cite[Section 4]{taud2} for several relationships which link the concepts of equal distribution, equal localization, spectral distribution, spectral clustering, etc.
\end{Remark}

\begin{Remark}\label{Toeplitz_distribution}
Since it was proved in \cite{Tillinota} that $\{T_n(f)\}_{n\in\mathbb N^k}\sim_\lambda(f,I_k)$ for $f\in L^1(k,s)$, every matrix-sequence $\{T_{n(m)}(f)\}_m$ such that $\min_jn_j(m)\to\infty$ is weakly clustered at $\mathcal{ER}(f)$ in the sense of the eigenvalues.
\end{Remark}

\begin{Definition}\label{area}
{\rm Let $K$ be a compact subset of ${\mathbb C}$.
We define
\begin{equation*}
Area(K):=\mathbb C\backslash U,
\end{equation*}
where $U$ is the (unique) unbounded connected component of ${\mathbb C}\backslash K$.}
\end{Definition}

Now we are ready for stating the main tool that we shall use for the proof of our main result (Theorem \ref{szego-nonherm-prec}).

\begin{Theorem}\label{mergelyan-cons-improved-2}
{\rm\cite{maya}} Let $\{A_n\}$ be a matrix-sequence, with $A_n$ of size $d_n$ tending to infinity. If:
\begin{itemize}
	\item[$\boldsymbol{(c_1)}$] the spectrum $\Lambda_n$ of $A_n$ is uniformly bounded, i.e., $|\lambda|<C$ for all $\lambda\in\Lambda_n$, for all $n$, and for some constant $C$ independent of $n$;
	\item[$\boldsymbol{(c_2)}$] there exists a measurable function $h\in L^\infty(k,s)$, defined over a certain domain $G\subset\mathbb R^k$ of finite and positive Lebesgue measure, such that, for every non-negative integer $N$, we have
	$$\lim_{n\rightarrow \infty}\frac{{\rm tr}(A_n^N)}{d_n}=\frac1{m_k(G)}\int_G\frac{{\rm tr}(h^N(t))}{s}\,dt;$$
	\item[$\boldsymbol{(c_3)}$] $\{P(A_n)\}\sim_\sigma (P(h),G)$ for every polynomial $P$;
\end{itemize}
then the matrix-sequence $\{A_n\}$ is weakly clustered at $Area({\cal ER}(h))$ and relation \eqref{distribution:sv-eig} is true for every $F\in\mathcal C_0(\mathbb C)$ which is holomorphic in the interior of $Area({\cal ER}(h))$.
If moreover:
\begin{itemize}
	\item[$\boldsymbol{(c_4)}$] $\mathbb C\backslash{\cal ER}(h)$ is connected and the interior of ${\cal ER}(h)$ is empty;
\end{itemize}
then $\{A_n\}\sim_\lambda(h,G)$.
\end{Theorem}

Using Theorem \ref{mergelyan-cons-improved-2}, in \cite{maya} the authors proved the following result.

\begin{Theorem}\label{szego-nonherm}
{\rm \cite{maya}} Let $f\in L^\infty(k,s)$. If ${\cal ER}(f)$ has empty interior and does not disconnect the complex plane, then $\{T_n(f)\}_{n\in\mathbb N^k}\sim_\lambda(f,I_k).$
\end{Theorem}

Theorem \ref{szego-nonherm} generalizes to the matrix-valued case a result by Tilli \cite{tillicomplex}, holding in the scalar-valued case $s=1$.
Here, we generalize Theorem \ref{szego-nonherm}, which concerns the non-preconditioned matrix-family $\{T_n(f)\}_{n\in\mathbb N^k}$, to the case of preconditioned matrix-families of the form $\{T_n^{-1}(g)T_n(f)\}_{n\in\mathbb N^k}$. Note that, for a function $g\in L^\infty(k,s)$, the essential numerical range $\mathcal{ENR}(g)$ is compact and hence ${\rm Coh}[{\cal ENR}(g)]$ is also compact (we recall that the convex hull of a compact set is compact). Therefore, if $g\in L^\infty(k,s)$ and $0\notin{\rm Coh}[{\cal ENR}(g)]$, then $g$ is sectorial (Lemma \ref{lemma_geometrico}) and $T_n(g)$ is invertible for all $n\in\mathbb N^k$ (Theorem \ref{locsv4}). The condition $0\notin{\rm Coh}[\mathcal{ENR}(g)]$ also ensures that $g$ is invertible a.e., because, for almost every $x\in I_k$, $\lambda_i(g(x))\in\mathcal{ER}(g)\subseteq\mathcal{ENR}(g)\subseteq{\rm Coh}[\mathcal{ENR}(g)]$ for all $i=1,\ldots,s$, implying that $\lambda_i(g)\ne0$ for all $i=1,\ldots,s,$ a.e.

\begin{Theorem}\label{szego-nonherm-prec}
Let $f,g\in L^\infty(k,s)$, with $0\notin{\rm Coh}[{\cal ENR}(g)]$, and let $h:=g^{-1}f$.
If ${\cal ER}(h)$ has empty interior and does not disconnect the complex plane, then $\{T_n^{-1}(g)T_n(f)\}_{n\in\mathbb N^k}\sim_\lambda(h,I_k)$.
\end{Theorem}

Before beginning with the proof of Theorem \ref{szego-nonherm-prec}, some preliminary work is needed. Given a square matrix $A$ of size $m$, we denote by $\|A\|_{(p)}$ the $p$-norm of $A$, that is the $p$-norm of the vector of length $m^2$ obtained by putting all the columns of $A$ one below the other. The notation $\|A\|_p$ is reserved for the Schatten $p$-norm of $A$, defined as the $p$-norm of the vector formed by the singular values of $A$. In symbols, $\|A\|_p=(\sum_{j=1}^m\sigma_j^p(A))^{1/p}$ for $1\le p<\infty$ and $\|A\|_\infty=\max_{j=1,\ldots,m}\sigma_j(A)=\|A\|$. The Schatten 1-norm is also called the trace-norm. We refer the reader to \cite{Bhatia} for the properties of the Schatten $p$-norms. We only recall from \cite[Problem III.6.2 and Corollary IV.2.6]{Bhatia} the H\"older inequality $\|AB\|_1\leq\|A\|_p\|B\|_q$, which is true for all square matrices $A,B$ of the same size and whenever $p,q\in[1,\infty]$ are conjugate exponents (i.e. $\frac1p+\frac1q=1$). In particular, we will need the H\"older inequality with $p=1$ and $q=\infty$, which involves the spectral norm and the trace-norm:
\begin{equation}\label{p=1,q=infinito}
\|AB\|_1\le\|A\|\|B\|_1.
\end{equation}

Now, let us show that, for any $1\le p\le\infty$, $L^p(k,s)=L^p(I_k,dx,\mathcal M_s)$, where
\begin{align*}
L^p(I_k,dx,\mathcal M_s)&:=\left\{f:I_k\to\mathcal M_s\left|\,\mbox{$f$ is measurable, } \int_{I_k}\|f(x)\|_p^pdx<\infty\right.\right\},&\mbox{if $1\le p<\infty$,}\\
L^\infty(I_k,dx,\mathcal M_s)&:=\left\{f:I_k\to\mathcal M_s\left|\,\mbox{$f$ is measurable, }\mathop{\rm ess\,sup}_{x\in I_k}\|f(x)\|_\infty<\infty\right.\right\}.
\end{align*}
Since $\mathcal M_s$ is a finite-dimensional vector space, all the norms on $\mathcal M_s$ are equivalent. In particular, $\|\cdot\|_{(p)}$ and $\|\cdot\|_p$ are equivalent, and so there are two positive constants $\alpha,\beta$ such that
$$\alpha\|f(x)\|_p\le\|f(x)\|_{(p)}\le\beta\|f(x)\|_p,\quad\forall x\in I_k.$$
It follows that
\begin{align}
&\alpha^p\int_{I_k}\|f(x)\|_p^pdx\le\int_{I_k}\|f(x)\|_{(p)}^pdx\le\beta^p\int_{I_k}\|f(x)\|_p^pdx, & \mbox{if $1\le p<\infty$,}\label{Lplp}\\
&\alpha\mathop{\rm ess\,sup}_{x\in I_k}\|f(x)\|_\infty\le\mathop{\rm ess\,sup}_{x\in I_k}\|f(x)\|_{(\infty)}\le\beta\mathop{\rm ess\,sup}_{x\in I_k}\|f(x)\|_\infty.\label{oo}
\end{align}
Therefore, if $f\in L^p(k,s)$ then each component $f_{ij}:I_k\to\mathbb C,\ i,j=1,\ldots,s,$ belongs to $L^p(I_k)$ and the first inequalities in \eqref{Lplp}--\eqref{oo} says that $f\in L^p(I_k,dx,\mathcal M_s)$. Conversely, if $f\in L^p(I_k,dx,\mathcal M_s)$, the second inequalities in \eqref{Lplp}--\eqref{oo} says that $f\in L^p(k,s)$. This concludes the proof of the identity $L^p(k,s)=L^p(I_k,dx,\mathcal M_s)$ and allows us to define the following functional norm on $L^p(k,s)$:
$$\|f\|_{L^p}:=\left\|\,\|f(x)\|_p\,\right\|_{L^p(I_k)}=\left\{\begin{array}{ll}
\left(\int_{I_k}\|f(x)\|_p^pdx\right)^{1/p}, & \mbox{if $1\le p<\infty$,}\\
\mathop{\rm ess\,sup}_{x\in I_k}\|f(x)\|_\infty, & \mbox{if $p=\infty$.}
\end{array}\right.$$
If $p,q\in[1,\infty]$ are conjugate exponents and $f\in L^p(k,s),\ g\in L^q(k,s)$, then a computation involving the H\"older inequalities for both Schatten $p$-norms and $L^p(I_k)$-norms shows that $fg\in L^1(k,s)$ and, in fact, $\|fg\|_{L^1},\|gf\|_{L^1}\leq\|f\|_{L^p}\|g\|_{L^q}.$ In particular, we will need the inequality with $p=1$ and $q=\infty$, i.e.
\begin{equation}\label{Holder}
\|fg\|_{L^1},\|gf\|_{L^1}\le\|f\|_{L^1}\|g\|_{L^\infty}.
\end{equation}
We also recall some known facts concerning the spectral norm and the Schatten 1-norm of Toeplitz matrices, see \cite[Corollary 3.5]{ssMore}:
\begin{align}
\label{L1}f\in L^1(k,s)\quad\Rightarrow\quad\|T_{n}(f)\|_{1}&\leq\frac{\widehat{n}}{(2\pi)^{k}}\|f\|_{L^{1}},\quad\forall n\in\mathbb N^k;\\
\label{Linfty}f\in L^\infty(k,s)\quad\Rightarrow\quad\|T_{n}(f)\|&\leq\|f\|_{L^{\infty}},\quad\forall n\in\mathbb N^k.
\end{align}

In order to prove Theorem \ref{szego-nonherm-prec}, we still need two results. The first (Proposition \ref{PropRank}) provide an estimate of the rank of $T_n(g)T_n(f)-T_n(gf)$, in the case where $f\in L^1(k,s)$ and $g$ is a $k$-variate trigonometric polynomial of degree $r=(r_1,\ldots,r_k)$ taking values in $\mathcal M_s$ (see Definition \ref{trig_poly}). The second result (Proposition \ref{TgTh-Tgh}) concerns the evaluation of the trace-norm of $T_n(g)T_n(f)-T_n(gf)$ for $f,g\in L^\infty(k,s)$, which is a crucial point for the proof of Theorem \ref{szego-nonherm-prec}. For $s=1$, we can find the proofs of these results (full for $k=1$ and sketched for $k>1$) in \cite{LaaSSS}. For completeness, we report the full proofs for $k>1$, also considering the generalization to $s>1$.

We recall that $g:\mathbb C^k\to\mathbb C$ is a $k$-variate trigonometric polynomial if $g$ is a finite linear combination of the $k$-variate functions (Fourier frequences) $\{{\rm e}^{{\mathbf i}\langle j,x\rangle}:j\in\mathbb Z^k\}$. Therefore, if $g$ is a $k$-variate trigonometric polynomial, then $g$ has only a finite number of nonzero Fourier coefficients $\widehat g_j$ and we define the degree $r=(r_1,\ldots,r_k)$ of $g$ as follows: for each $i=1,\ldots,k,$ $r_i$ is the maximum of $|j_i|,$ where $j=(j_1,\ldots,j_k)$ varies among all multi-indices in $\mathbb Z^k$ such that $\widehat g_j\ne0$ ($r_i$ is called the degree of $g(x)$ with respect to the $i$-th variable $x_i$). Observe that a $k$-variate trigonometric polynomial $g$ of degree $r=(r_1,\ldots,r_k)$ can be written in the form $g(x)=\sum_{j=-r}^r\widehat g_j{\rm e}^{\mathbf i\langle j,x\rangle}.$

\begin{Definition}\label{trig_poly}
We say that $g:\mathbb C^k\rightarrow{\cal M}_s$ is a $k$-variate trigonometric polynomial if, equivalently:
\begin{itemize}
	\item all the components $g_{l,t}:\mathbb C^k\to\mathbb C,\ l,t=1,\ldots,s,$ are $k$-variate trigonometric polynomials.
	\item $g$ is a finite linear combination (with coefficients in $\mathcal M_s$) of the $k$-variate functions $\{{\rm e}^{{\mathbf i}\langle j,x\rangle}:j\in\mathbb Z^k\}$.
\end{itemize}
If $g$ is a $k$-variate trigonometric polynomial, then $g$ has only a finite number of nonzero Fourier coefficients $\widehat g_j\in\mathcal M_s$ and the degree $r=(r_1,\ldots,r_k)$ of $g$ is defined in two equivalent ways:
\begin{itemize}
	\item for each $i=1,\ldots,k,$ $r_i$ is the maximum degree among all the polynomials $g_{l,t}(x)$ with respect to the $i$-th variable $x_i$;
	\item for each $i=1,\ldots,k,$ $r_i$ is the maximum of $|j_i|,$ where $j=(j_1,\ldots,j_k)$ varies among all multi-indices in $\mathbb Z^k$ such that $\widehat g_j$ is nonzero.
\end{itemize}
We note that a $k$-variate trigonometric polynomial $g$ of degree $r=(r_1,\ldots,r_k)$ can be written in the form $g(x)=\sum_{j=-r}^r\widehat g_j{\rm e}^{\mathbf i\langle j,x\rangle},$ where the Fourier coefficients $\widehat g_j$ belong to $\mathcal M_s$.
\end{Definition}

\begin{Proposition}\label{PropRank}
Let $f,g\in L^1(k,s)$, with $g$ a $k$-variate trigonometric polynomial of degree $r=(r_1,\ldots,r_k)$, and let $n$ be a $k$-index such that $n\ge 2r+e$. Then
\begin{equation}\label{rank}
{\rm rank}(T_n(g)T_n(f)-T_n(gf))\le s\left[\widehat n-\prod_{i=1}^k(n_i-2r_i)\right].
\end{equation}
\end{Proposition}
\proof
Since $g:\mathbb C^k\to\mathcal M_s$ is a $k$-variate trigonometric polynomial of degree $r$, we can write $g$ in the form
\begin{equation*}
g(x)=\sum_{j=-r}^{r}\widehat g_j{\rm e}^{\mathbf i\left\langle j,x\right\rangle}.
\end{equation*}
The Fourier coefficients of $(gf)(x)=g(x)f(x)$ are given by
$$(\widehat{gf})_\ell=\frac1{(2\pi)^k}\int_{I_k}g(x)f(x){\rm e}^{-\mathbf i\langle \ell,x\rangle}dx=\sum_{j=-r}^r\widehat g_j\frac1{(2\pi)^k}\int_{I_k}f(x){\rm e}^{-\mathbf i\langle\ell-j,x\rangle}dx=\sum_{j=-r}^r\widehat g_j\widehat f_{\ell-j}.$$
Now, using the definition of multilevel block Toeplitz matrices, see \eqref{mbTm_multi-index}, for all $l,t=e,\ldots,n$ we have
\begin{equation}\label{toeplitzprodotto}
T_n(gf)_{l,t}=(\widehat{gf})_{l-t}=\sum_{j=-r}^r\widehat g_j\widehat f_{l-t-j},
\end{equation}
and
\begin{equation}\label{prodottotoeplitz}
(T_n(g)T_n(f))_{l,t}=\sum_{v=e}^nT_n(g)_{l,v}T_n(f)_{v,t}=\sum_{v=e}^n\widehat g_{l-v}\widehat f_{v-t}=\sum_{j=l-n}^{l-e}\widehat g_j\widehat f_{l-j-t}=\sum_{j=\max(l-n,-r)}^{\min(l-e,r)}\widehat g_j\widehat f_{l-t-j},
\end{equation}
where the last equality is motivated by the fact that $\widehat g_j$ is zero if $j<-r$ or $j>r$. Therefore, (\ref{toeplitzprodotto}) and (\ref{prodottotoeplitz}) coincide when $r+e\leq l \leq n-r$. Observe that the multi-index range $r+e,\ldots, n-r$ is nonempty because of the assumption $n\ge 2r+e$. We conclude that the only possible nonzero rows of $T_n(g)T_n(f)-T_n(gf)$ are those corresponding to multi-indices $l$ in the set $\{e,\ldots,n\}\backslash\{r+e,\ldots,n-r\}$. This set has cardinality $\widehat n-\prod_{i=1}^k(n_i-2r_i)$ and so $T_n(g)T_n(f)-T_n(gf)$ has at most $\widehat n-\prod_{i=1}^k(n_i-2r_i)$ nonzero rows. Now we should notice that each row of $T_n(g)T_n(f)-T_n(gf)$ is actually a block-row of size $s$, i.e., a $s\times s\widehat n$ submatrix of $T_n(g)T_n(f)-T_n(gf)$. Indeed, each component of $T_n(f)$, $T_n(g)$, $T_n(g)T_n(f)-T_n(gf)$ is actually a $s\times s$ matrix, see \eqref{mbTm_multi-index}. Therefore, the actual nonzero rows of $T_n(g)T_n(f)-T_n(gf)$ are at most $s[\widehat n-\prod_{i=1}^k(n_i-2r_i)]$ and \eqref{rank} is proved.
\endproof

\begin{Proposition}\label{TgTh-Tgh}
Let $f,g\in L^{\infty}(k,s)$, then $\|T_{n}(g)T_{n}(f)-T_{n}(gf)\|_{1}=o(\widehat n)$ as $n\to\infty$.
\end{Proposition}
\proof
Let $g_m:\mathbb{C}^{k}\rightarrow{\cal M}_{s},\ g_m=[{(g_m)}_{l,t}]_{l,t=1}^s,$ be a $k$-variate trigonometric polynomial of degree $m=(m_{1},\ldots,m_{k})$. Let $m^{-}:=(m_{1}^{-},\ldots,m_{k}^{-})$, where $m_i^{-}$ is the minimum degree among all the polynomials ${(g_m)}_{l,t}(x)$ with respect to the variable $x_i$. We choose $g_{m}$ such that $\|g_m\|_{L^{\infty}}\leq\|g\|_{L^{\infty}}$ for every $m$ and $\|g_{m}-g\|_{L^{1}}\to0$ as $m^-\to\infty$. The polynomials $g_m$ can be constructed by using the $m$-th Cesaro sum of $g$ (see \cite{Zygmund}) and indeed the linear positive character of the Cesaro operator and Korovkin theory \cite{Korovkin,KoroSINUM} imply the existence of a $g_m$ with the desired properties.
Note that, by \eqref{Lplp} with $p=1$, the fact that $\|g_m-g\|_{L^1}\to0$ as $m^-\to\infty$ is equivalent to saying that $\|{(g_m)}_{l,t}-g_{l,t}\|_{L^{1}}\rightarrow0$ as $m^{-}\rightarrow\infty$ for all $l,t=1,\ldots,s$. Now, by adding and subtracting and by using the triangle inequality several times we get
\begin{align}\label{NormDifference}
\notag&\|T_{n}(g)T_{n}(f)-T_{n}(gf)\|_{1}\\
&\leq\|T_{n}(g)T_{n}(f)-T_{n}(g_m)T_{n}(f)\|_{1}+\|T_{n}(g_m)T_{n}(f)-T_{n}(g_{m}f)\|_{1}+\|T_{n}(g_{m}f)-T_{n}(gf)\|_{1}.
\end{align}
Using the linearity of the operator $T_{n}(\cdot)$, the H\"older inequality \eqref{p=1,q=infinito} and \eqref{Holder}--(\ref{Linfty}), we obtain
\begin{align}
\label{delta1}\|T_{n}(g)T_{n}(f)-T_{n}(g_m)T_{n}(f)\|_{1}&\leq\|T_{n}(g-g_m)\|_{1}\|T_{n}(f)\|\leq\frac{\widehat{n}}{(2\pi)^{k}}\|g_{m}-g\|_{L^{1}}\|f\|_{L^{\infty}}\\
\label{delta3}\|T_{n}(g_{m}f)-T_{n}(gf)\|_{1}&\leq\frac{\widehat{n}}{(2\pi)^{k}}\|g_{m}f-gf\|_{L^{1}}\leq\frac{\widehat{n}}{(2\pi)^{k}}\|g_{m}-g\|_{L^{1}}\|f\|_{L^{\infty}}.
\end{align}
Moreover, using the relation $\|A\|_{1}\leq{\rm rank}(A)\|A\|$ for a square matrix $A$ and the inequality $(1+c)^k\ge1+kc$ for $c\ge-1$, and setting $\|m\|_{\infty}:=\max_{j}m_{j}$, Proposition \ref{PropRank} tells us that, for any $n\ge 2m+e$,
\begin{align}
\notag\|T_{n}(g_m)T_{n}(f)-T_{n}(g_{m}f)\|_{1}&\leq{\rm rank}(T_{n}(g_m)T_{n}(f)-T_{n}(g_{m}f))\|T_{n}(g_m)T_{n}(f)-T_{n}(g_{m}f)\|\\
\notag&\le s\widehat n\left[1-\prod_{i=1}^k\left(1-\frac{2m_i}{n_i}\right)\right](\|T_{n}(g_m)T_{n}(f)\|+\|T_{n}(g_{m}f)\|)\\
\notag&\le s\widehat n\left[1-\left(1-\frac{2\|m\|_\infty}{\min_jn_j}\right)^k\right](2\|g_m\|_{L^{\infty}}\|f\|_{L^{\infty}})\\
\label{delta2}&\le s\widehat nk\frac{2\|m\|_\infty}{\min_jn_j}(2\|g\|_{L^{\infty}}\|f\|_{L^{\infty}}) =4sk\|m\|_{\infty}\|g\|_{L^{\infty}}\|f\|_{L^{\infty}}\frac{\widehat{n}}{\min_jn_j}.
\end{align}
Substituting (\ref{delta1})--(\ref{delta2}) in (\ref{NormDifference}), for each $k$-tuple $m$ and for each $n\ge2m+e$ the following inequality holds:
\begin{align*}
\|T_{n}(g)T_{n}(f)-T_{n}(gf)\|_{1}\leq\widehat{n}\xi(m)+\gamma(m)\frac{\widehat n}{\min_jn_j},
\end{align*}
where $\xi(m):=2(2\pi)^{-k}\|g_{m}-g\|_{L^{1}}\|f\|_{L^{\infty}}$, $\gamma(m):=4sk\|m\|_{\infty}\|g\|_{L^{\infty}}\|f\|_{L^{\infty}}$, and we note that $\xi(m)\rightarrow 0$ as $m^{-}\rightarrow\infty$. Now, for $\epsilon>0$, we choose a $k$-tuple $m$ such that $\xi(m)<\epsilon/2$. For $n\to\infty$ (i.e. for $\min_jn_j\to\infty$) we have $\gamma(m)/\min_jn_j\to0$ and so we can choose a $\nu\ge2\|m\|_\infty+1$ such that $\gamma(m)/\min_jn_j\le\epsilon/2$ for $\min_jn_j\ge\nu$. Then, if $\min_jn_j\ge\nu$, we have $n\ge2m+e$ and
$$
\frac{\|T_{n}(g)T_{n}(f)-T_{n}(gf)\|_{1}}{\widehat{n}}\leq\epsilon.
$$
This means that $\dfrac{\|T_{n}(g)T_{n}(f)-T_{n}(gf)\|_{1}}{\widehat n}\to0$ as $n\to\infty$, i.e. $\|T_{n}(g)T_{n}(f)-T_{n}(gf)\|_{1}=o(\widehat{n})$ as $n\to\infty$.
\endproof

The following Lemma is the last result that we need for the proof of Theorem \ref{szego-nonherm-prec}. It shows that the function $h=g^{-1}f$ appearing in Theorem \ref{szego-nonherm-prec} belongs to $L^\infty(k,s)$.

\begin{Lemma}\label{hLoo}
Let $f,g\in L^\infty(k,s)$ with $0\notin{\rm Coh}[\mathcal{ENR}(g)]$, as in Theorem \ref{szego-nonherm-prec}. Then $h:=g^{-1}f\in L^\infty(k,s)$.
\end{Lemma}
\begin{proof}
Since $g\in L^\infty(k,s)$ and $0\notin{\rm Coh}[\mathcal{ENR}(g)]$, the convex hull ${\rm Coh}[{\cal ENR}(g)]$ is compact, the distance $d:=d({\rm Coh}[\mathcal{ENR}(g)],0)$ is positive, and $g$ is invertible a.e. (recall the discussion before the statement of Theorem \ref{szego-nonherm-prec}). We are going to show that
\begin{equation}\label{g-1}
\|g^{-1}(x)\|\le\frac1d,\quad\mbox{for a.e. $x\in I_k$.}
\end{equation}
Since in a matrix the absolute value of each component is bounded from above by the spectral norm, once we have proved \eqref{g-1}, it follows that $g^{-1}\in L^\infty(k,s)$, and the Lemma is proved. Now, by the fact that $d>0$ and by Lemma \ref{lemma_geometrico}, $g$ is sectorial. By Theorem \ref{locsv4}, first item, there exists a separating line $z$ for $\mathcal{ENR}(g)$ such that $d(z,0)=d$. Let $H_1$ be the open half-plane determined by $z$ satisfying $H_1\bigcap\mathcal{ENR}(g)=\emptyset$ and $0\in H_1\bigcup z$, and let $\omega(z)$ be the rotation number (of modulus 1) for which $\omega(z)\cdot z=\{w\in\mathbb C:{\rm Re}(w)=d(z,0)\}$. Then
\begin{equation}\label{enr_quadrato}
\mathcal{ENR}(\omega(z)g)=\omega(z)\cdot \mathcal{ENR}(g)\subseteq\{w\in\mathbb C:{\rm Re}(w)\ge d\}.
\end{equation}
Now observe that, for a.e. $x\in I_k$, $v^*[\omega(z)g(x)]v\in\mathcal{ENR}(\omega(z)g)$ for all $v\in\mathbb C^s$ with $\|v\|=1$ (see the discussion after Definition \ref{er-enr}). Therefore, by \eqref{enr_quadrato}, for a.e. $x\in I_k$ we have
$$v^*{\rm Re}(\omega(z)g(x))v={\rm Re}(v^*[\omega(z)g(x)]v)\ge d, \quad\forall v\in\mathbb C^s\mbox{ with }\|v\|=1,$$
which implies, by the minimax principle \cite{Bhatia},
$$\lambda_{\min}({\rm Re}(\omega(z)g(x)))\ge d.$$
Hence, by the Fan-Hoffman theorem \cite[Proposition III.5.1]{Bhatia}, for a.e. $x\in I_k$ we have
$$\|g^{-1}(x)\|=\frac1{\sigma_{\min}(g(x))}=\frac1{\sigma_{\min}(\omega(z)g(x))}\le\frac1{\lambda_{\min}({\rm Re}(\omega(z)g(x)))}\le\frac1d,$$
and \eqref{g-1} is proved.
\end{proof}

We are now ready to prove Theorem \ref{szego-nonherm-prec}. We will show that, under the assumptions of Theorem \ref{szego-nonherm-prec}, the conditions $\boldsymbol{(c_1)}$--$\boldsymbol{(c_4)}$ of Theorem \ref{mergelyan-cons-improved-2} are met with $(h,G)=(g^{-1}f,I_k)$, for any matrix-sequence $\{T_{n(m)}^{-1}(g)T_{n(m)}(f)\}_m$ extracted from $\{T_n^{-1}(g)T_n(f)\}_{n\in\mathbb N^k}$ and such that $\min_jn_j(m)\to\infty$. Actually, to simplify the notation, we suppress the index $m$ and we will talk about a generic matrix-sequence $\{T_n^{-1}(g)T_n(f)\}$ such that $\min_jn_j\to\infty$, where it is understood the presence of an underlying index $m$.

\begin{proof}[Proof of Theorem \ref{szego-nonherm-prec}.]
As observed in the proof of Lemma \ref{hLoo}, $d:=d({\rm Coh}[\mathcal{ENR}(g)],0)$ is positive. Hence, by Theorem \ref{locsv4},
\begin{equation*}
\|T_n^{-1}(g)\|=\frac1{\sigma_{\min}(T_n(g))}\le\frac1d.
\end{equation*}
By hypothesis $f\in L^\infty(k,s)$ and by \eqref{Linfty}, it follows that
\[ \|T_n^{-1}(g)T_n(f)\| \le \|f\|_{L^\infty}/d, \]
so that requirement $\boldsymbol{(c_1)}$ in Theorem \ref{mergelyan-cons-improved-2} is satisfied.
Since $h\in L^\infty(k,s)$ (by Lemma \ref{hLoo}) and since $\mathcal{ER}(h)$ has empty interior and does not disconnect the complex plane (by hypothesis), $h$ satisfies the assumptions of Theorem \ref{szego-nonherm} and so $\{T_n(h)\}\sim_\lambda(h,I_k).$ Therefore, using the inequality $|{\rm tr}(A)|\leq\|A\|_1$ for a square matrix $A$ (see \cite[Theorem II.3.6, Eq.\,(II.23)]{Bhatia}), item $\boldsymbol{(c_2)}$ in Theorem \ref{mergelyan-cons-improved-2} is proved if we show that
\begin{equation}\label{sn}
\|(T_n^{-1}(g)T_n(f))^N-T_n(h)^N\|_1=o(\widehat n)
\end{equation}
for every non-negative integer $N$. If $N=0$ the result is trivial. For $N=1$, using Proposition \ref{TgTh-Tgh} we obtain
\begin{align*}
\|T_n^{-1}(g)T_n(f)-T_n(h)\|_1 & = \|T_n^{-1}(g)(T_n(f)-T_n(g)T_n(h))\|_1 \\
                               & \le \|T_n^{-1}(g)\| \|T_n(f)-T_n(g)T_n(h)\|_1 \\
                               & \le \frac{1}{d}\|T_n(f)-T_n(g)T_n(h)\|_1 = o(\widehat n),
\end{align*}
so (\ref{sn}) is satisfied and we can write $T_n^{-1}(g)T_n(f)=T_n(h)+R_{n}$ with $\|R_{n}\|_{1}=o(\widehat n)$. Using this, when $N\geq 2$ we have
$$(T_n^{-1}(g)T_n(f))^N=(T_n(h))^N+S_n,$$
where $S_n$ is the sum of all possible (different) combinations of products of $j$ matrices $T_n(h)$ and $\ell$ matrices $R_{n}$, with $j+\ell=N$, $j\neq N$. By using the H\"older inequality \eqref{p=1,q=infinito}, and taking into account that $R_n=T_n^{-1}(g)T_n(f)-T_n(h)$, for every summand $S$ of $S_n$ we have
\begin{align*}
\|S\|_{1} & \leq\|T_n(h)\|^{j}\|R_{n}\|^{\ell-1}\|R_{n}\|_{1}\\
          & \leq\|h\|_{L^\infty}^{j}(\|f\|_{L^\infty}/d+\|h\|_{L^\infty})^{\ell-1}o(\widehat n)\leq C o(\widehat n),
\end{align*}
where $C$ is some positive constant. So, since the number of summands in $S_n$ is finite, (\ref{sn}) holds for every positive integer $N$, and requirement $\boldsymbol{(c_2)}$ in Theorem \ref{mergelyan-cons-improved-2} is then satisfied.
Requirement $\boldsymbol{(c_3)}$ in Theorem \ref{mergelyan-cons-improved-2} is also satisfied, because the sequences of multilevel block Toeplitz matrices with $L^1(k,s)$ symbols belong to the GLT class together with their algebra (see Section 3.3.1 in \cite{glt-vs-fourier}).
Finally, by taking into account that ${\cal ER}(h)$ has empty interior and does not disconnect the complex plane, the last condition $\boldsymbol{(c_4)}$ in Theorem \ref{mergelyan-cons-improved-2} is met, and the application of Theorem \ref{mergelyan-cons-improved-2} shows that $\{T_n^{-1}(g)T_n(f)\}\sim_\lambda(h,I_k)$.
\end{proof}

\section{Some applications and numerical experiments}\label{num}

In this section we consider a list of numerical examples which cover different situations. The first subsection is devoted to examples that involve 1-level matrix-valued symbols, while the second contains 2-level examples.

\subsection{Univariate examples}

Fixed $s=2$ and $k=1$, we consider $f$ and $g$ of the form
\begin{equation}\label{fg}
\begin{split}
f(x) & = Q(x)A(x)Q(x)^{T} \\
g(x) & = Q(x)B(x)Q(x)^{T}
\end{split}
\end{equation}
where
\begin{equation*}
Q(x)=\left(
\begin{array}{cc}
\cos (x) & \sin (x) \\
-\sin (x) & \cos (x)
\end{array}
\right),
\end{equation*}
while $A(x)$ and $B(x)$ vary from case to case. For each example, we focus our attention on the spectral behavior of the matrices $T_n(f)$ for different sizes $n$ and on the solution of the associated linear system with a random right-hand side. From a computational point of view, to solve such systems, we apply (full or preconditioned) GMRES with tolerance $10^{-6}$ using the Matlab built-in \texttt{gmres} function.

\begin{figure}[htb]
\centering
\begin{minipage}[l]{.23\textwidth}
\includegraphics[width=3.9cm]{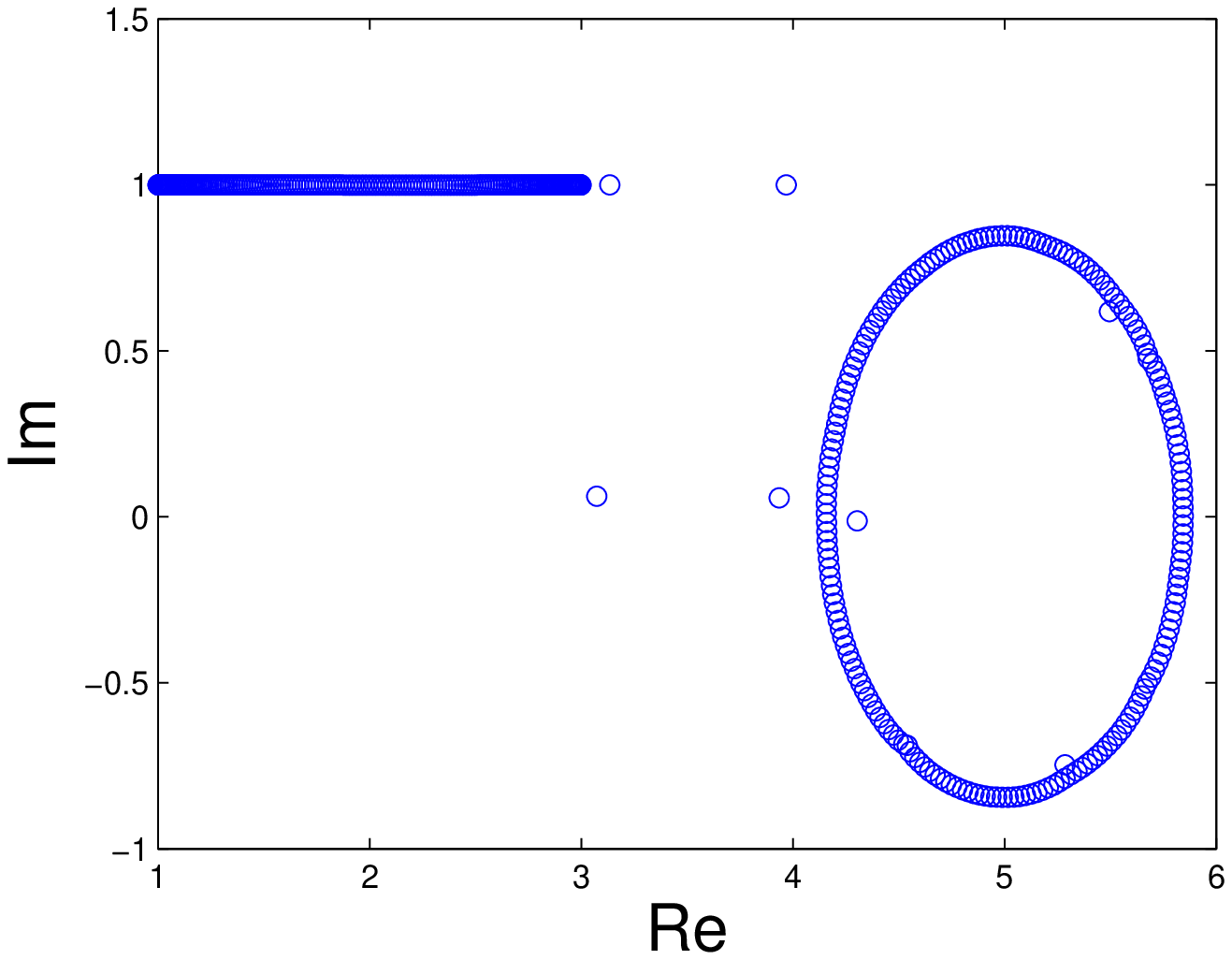}
\centering \footnotesize$r=1$
\end{minipage}
\begin{minipage}[l]{.23\textwidth}
\includegraphics[width=3.9cm]{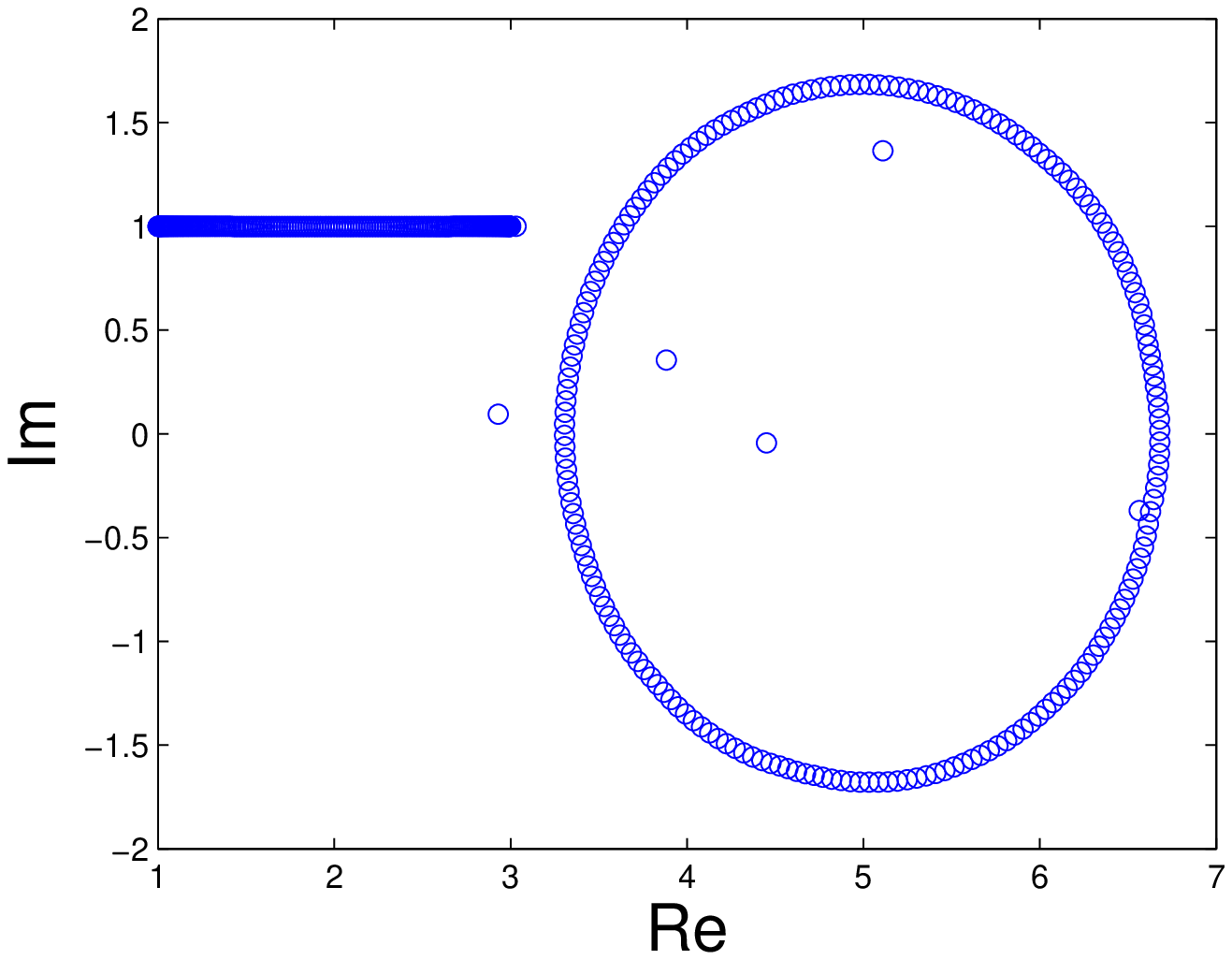}
\centering \footnotesize$r=2$
\end{minipage}
\begin{minipage}[l]{.23\textwidth}
\includegraphics[width=3.9cm]{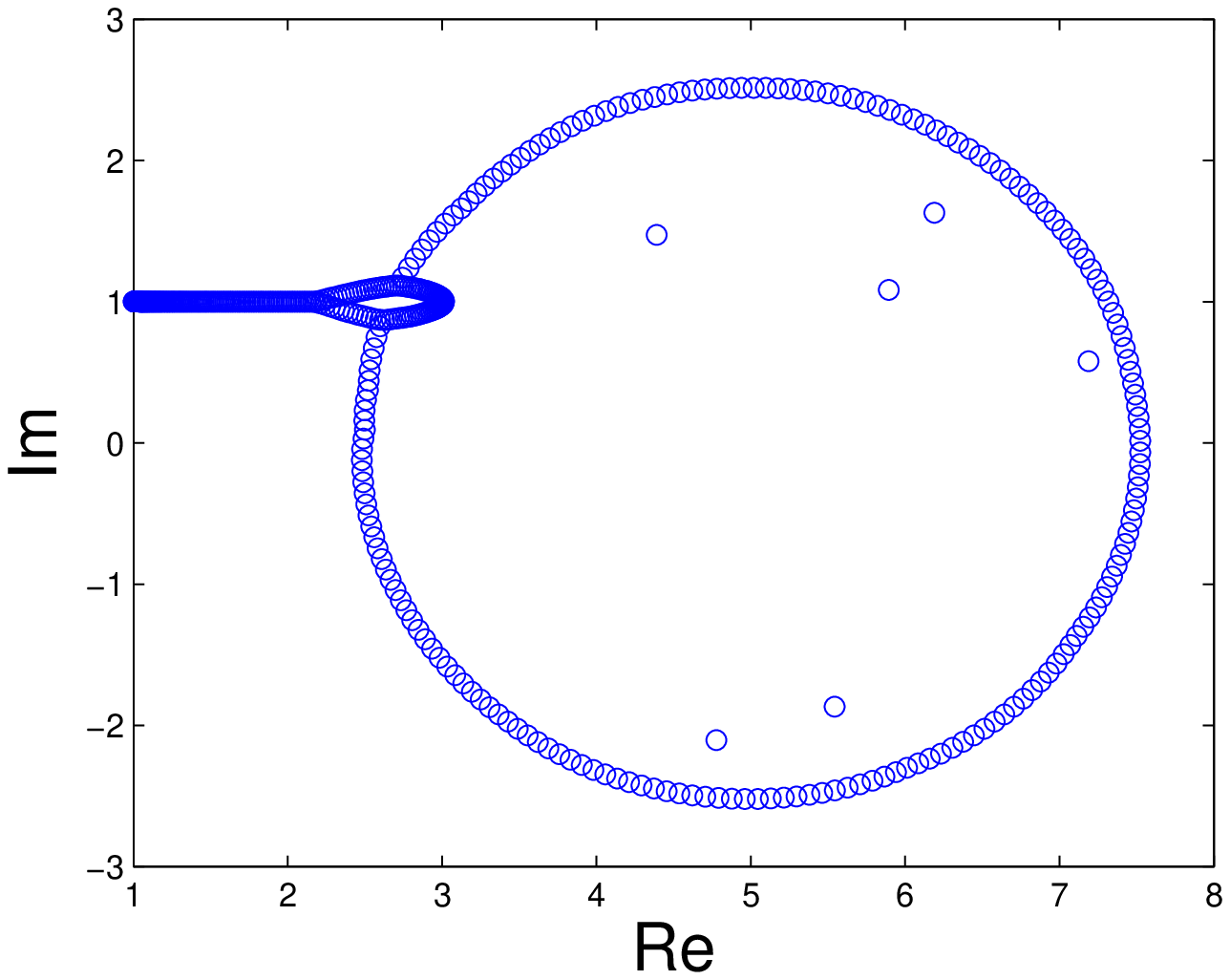}
\centering \footnotesize$r=3$
\end{minipage}
\begin{minipage}[l]{.23\textwidth}
\includegraphics[width=3.9cm]{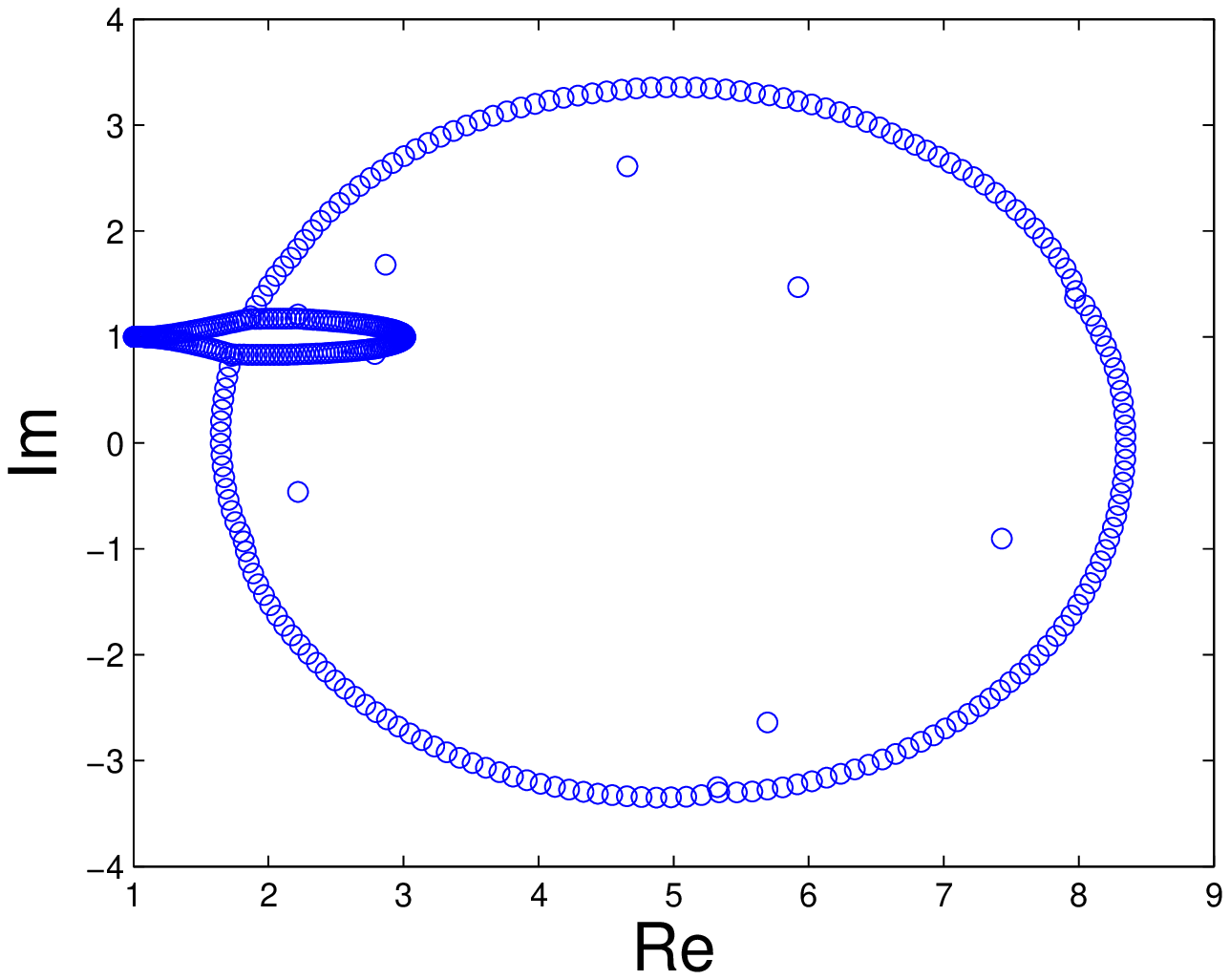}
\centering \footnotesize$r=4$
\end{minipage}
\caption{Eigenvalues in the complex plane
of $T_{200}(f_1)$ for $r=1,2,3,4$}\label{casotot}
\end{figure}

\begin{figure}[htb]
\centering
\begin{minipage}[l]{.23\textwidth}
\includegraphics[width=3.9cm]{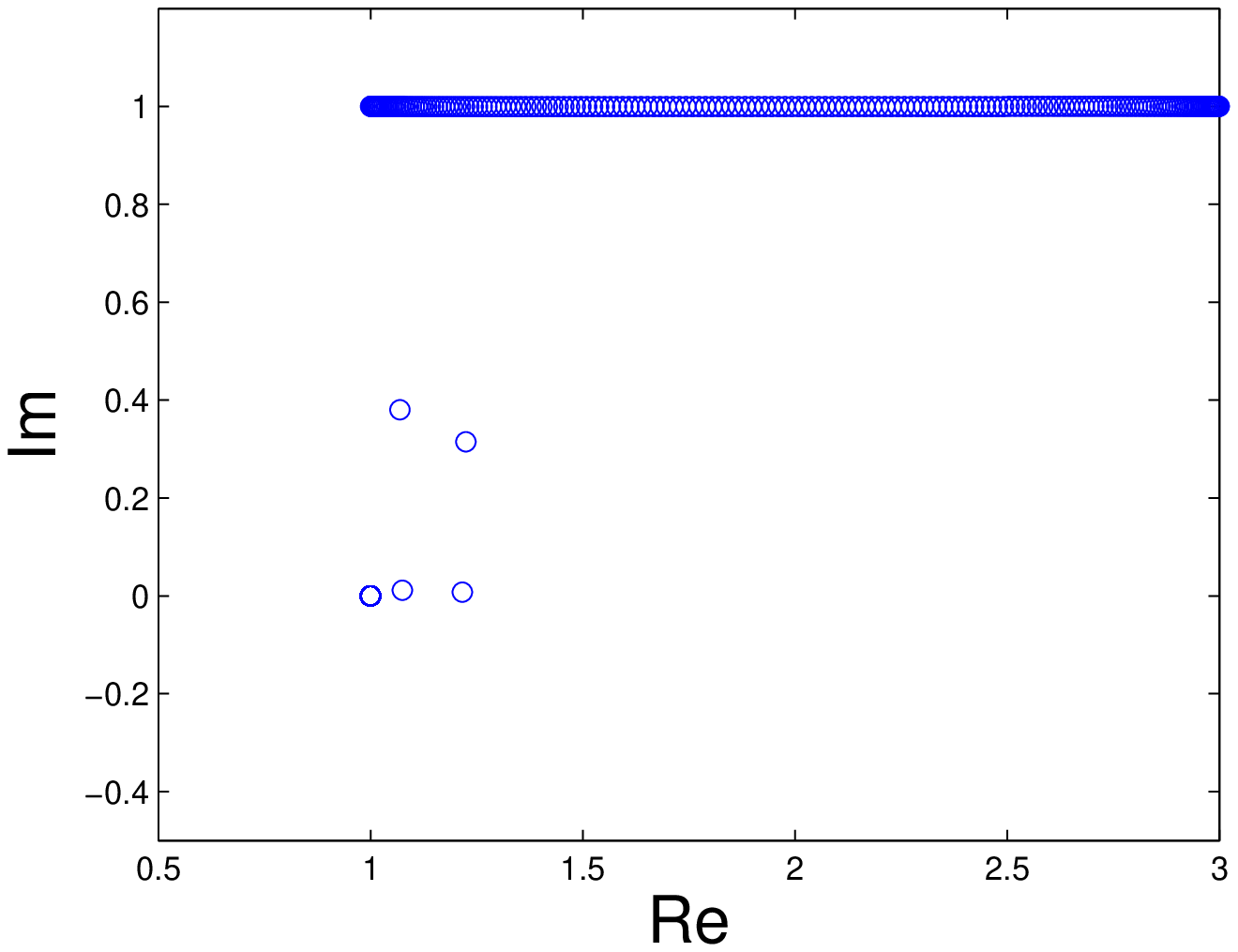}
\centering \footnotesize$r=1$
\end{minipage}
\begin{minipage}[l]{.23\textwidth}
\includegraphics[width=3.9cm]{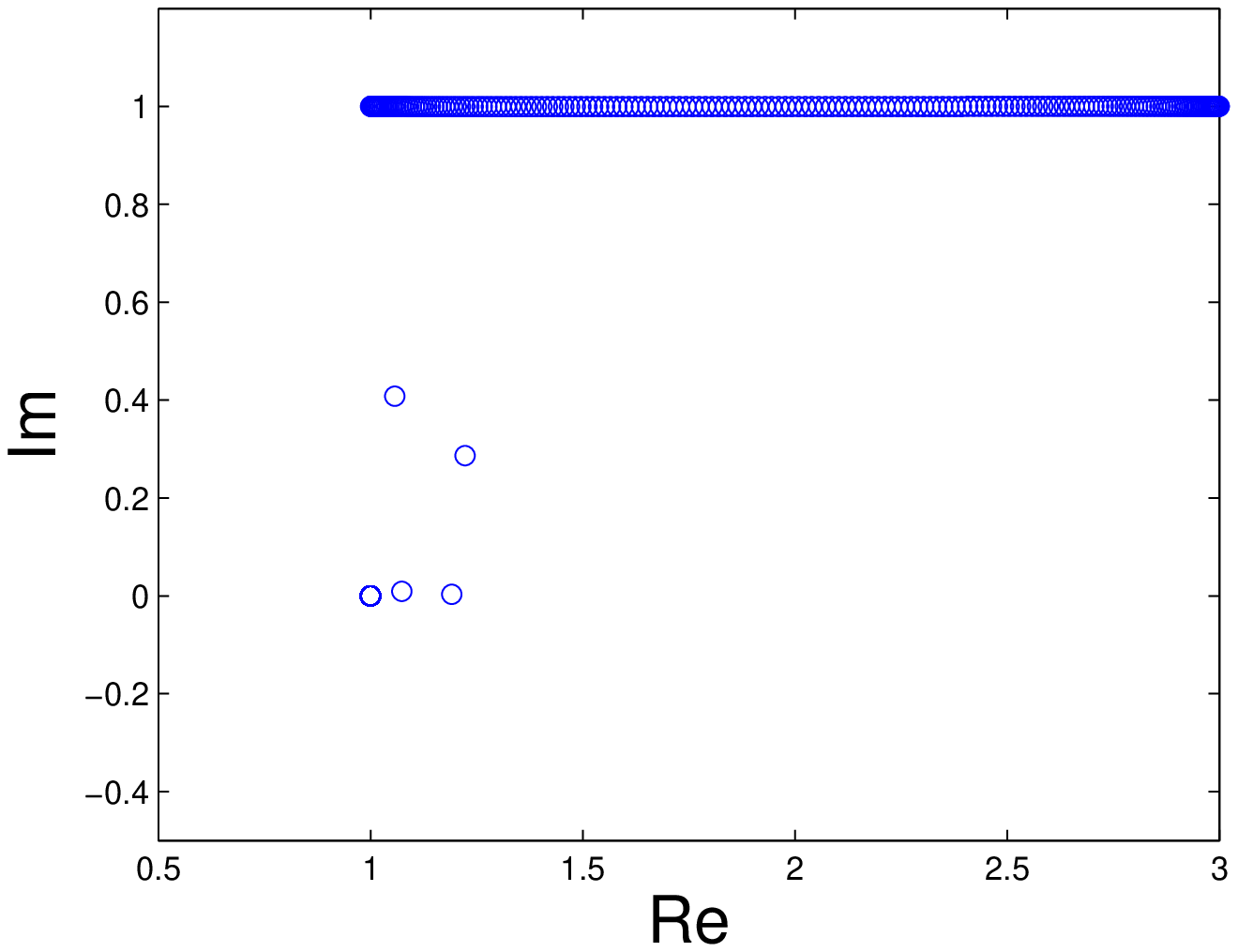}
\centering \footnotesize$r=2$
\end{minipage}
\begin{minipage}[l]{.23\textwidth}
\includegraphics[width=3.9cm]{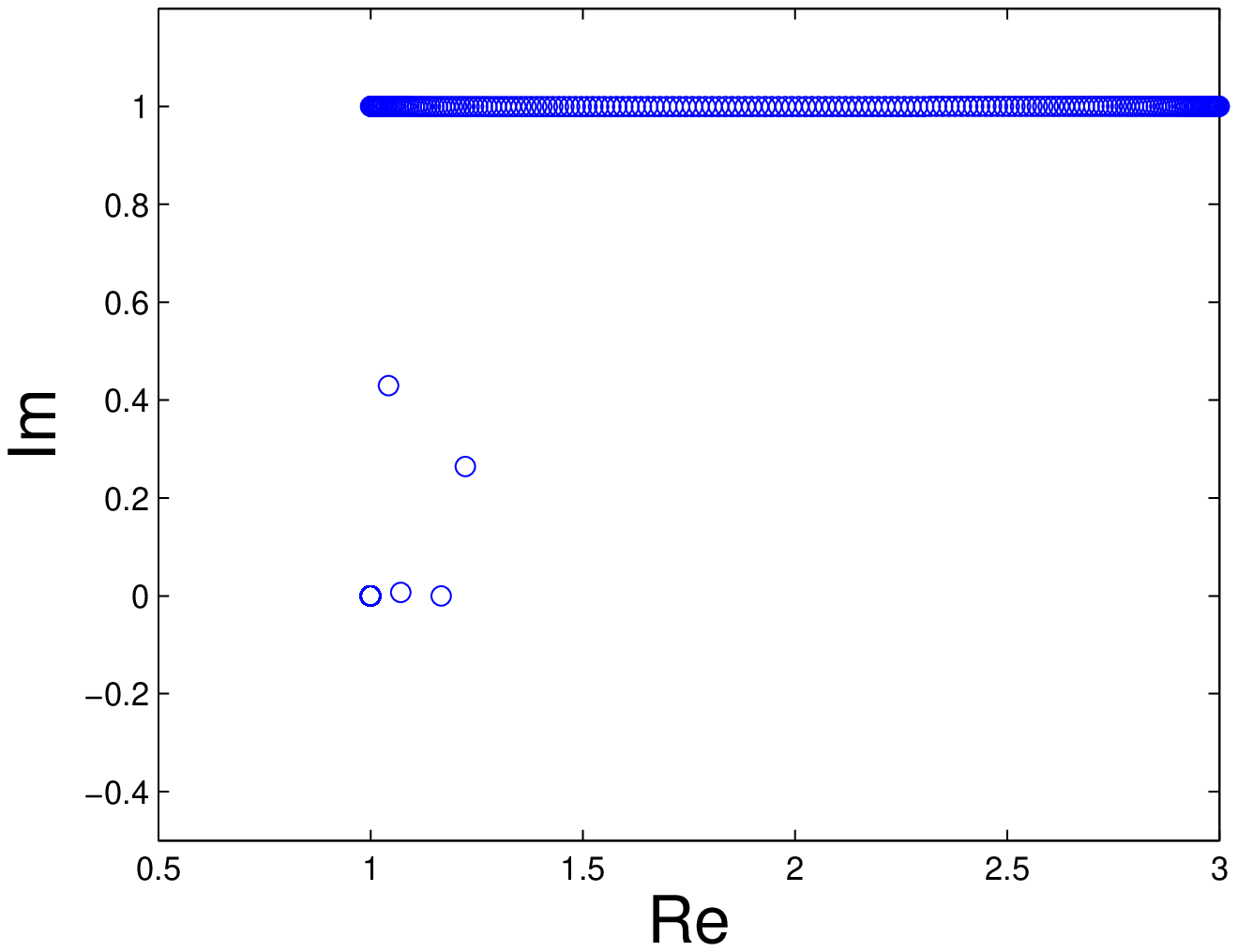}
\centering \footnotesize$r=3$
\end{minipage}
\begin{minipage}[l]{.23\textwidth}
\includegraphics[width=3.9cm]{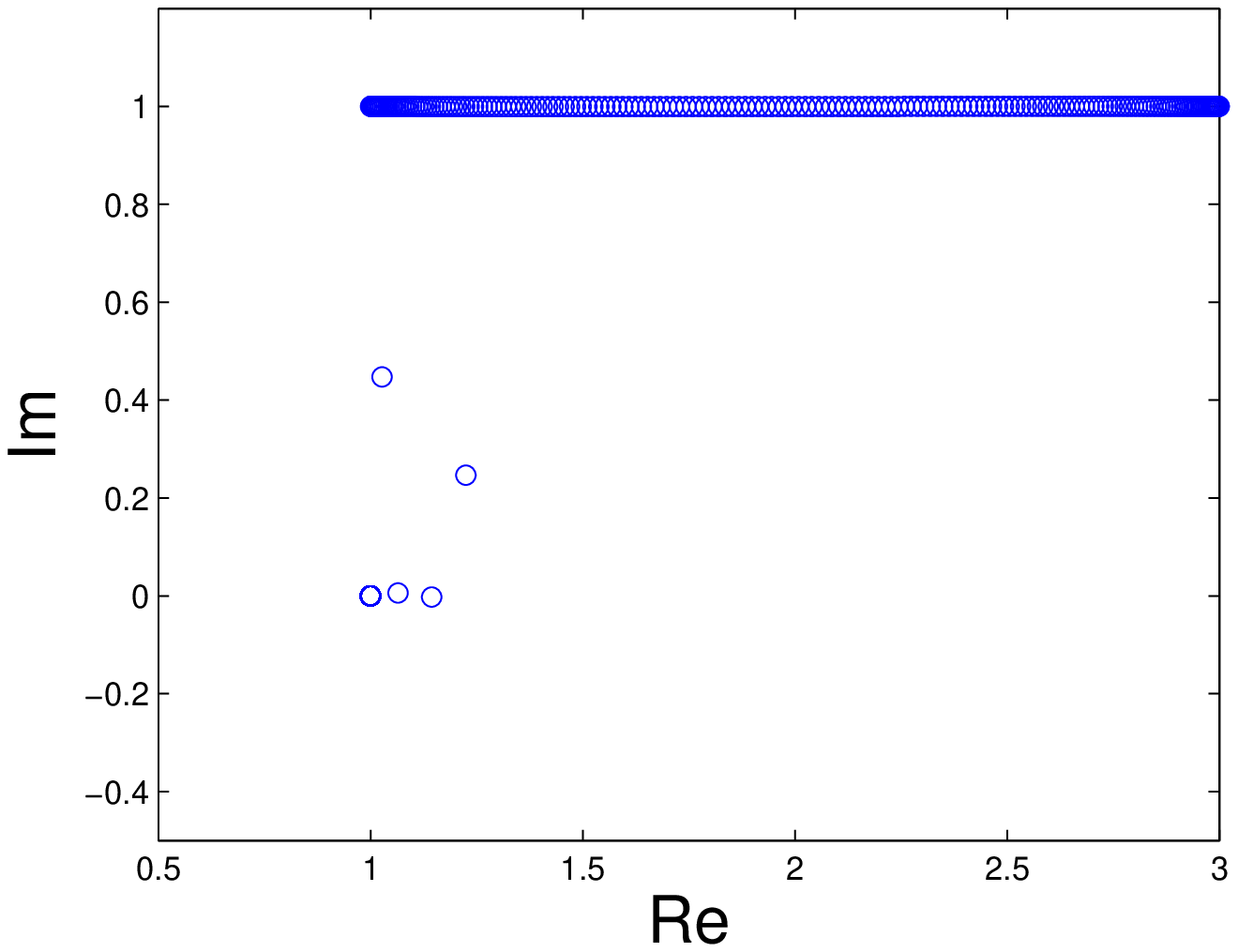}
\centering \footnotesize$r=4$
\end{minipage}
\caption{Eigenvalues in the complex plane
of $T^{-1}_{200}(g_1)T_{200}(f_1)$ for $r=1,2,3,4$}\label{casotot2}
\end{figure}

\begin{Case}\label{Case 1}
Let us choose $A^{(1)}(x)$ and $B^{(1)}(x)$ as follows
\begin{equation*}
\begin{array}{cc}
A^{(1)}(x)=\left(
\begin{array}{cc}
2+{\mathbf i}+\cos(x) & 0 \\
1 & 5+r{\rm e}^{{\mathbf i}x}
\end{array}
\right) , & B^{(1)}(x)=\left(
\begin{array}{cc}
1 & 0 \\
0 & 5+r{\rm e}^{{\mathbf i}x}
\end{array}
\right),
\end{array}
\end{equation*}
where $r$ is a real positive parameter, and define
\begin{align*}
f_{1}(x) & = Q(x)A^{(1)}(x)Q(x)^{T} \\
g_{1}(x) & = Q(x)B^{(1)}(x)Q(x)^{T}.
\end{align*}

Figures \ref{casotot} and \ref{casotot2} refer to the eigenvalues in the complex plane of $T_{200}(f_{1})$ and $T^{-1}_{200}(g_1)T_{200}(f_{1})$ for $r=1,2,3,4$. As expected (see Remark \ref{Toeplitz_distribution}), the eigenvalues of $T_{200}(f_{1})$  are distributed as $\lambda_1(f_1(x))=\lambda_1(A^{(1)}(x))=A^{(1)}_{1,1}(x)$ and $\lambda_2(f_1(x))=\lambda_2(A^{(1)}(x))=A^{(1)}_{2,2}(x)$, and in fact they are clustered at the union of the ranges of $\lambda_1(f_1(x))$ and $\lambda_2(f_1(x))$, which is the essential range of $f_1(x)$. More precisely, the matrix $T_{200}(f_{1})$ has two sub-clusters for the eigenvalues: one collects the eigenvalues with real part in $[1,3]$ and imaginary part around $1$ (such eigenvalues recall the behavior of the function $\lambda_1(f_1(x))$); the other, miming $\lambda_2(f_1(x))$, is made by a circle centered in $5$ with radius $r$, in agreement with theoretical results. We know that the GMRES in this case is optimal, since the eigenvalues of $f_1$ have no zeros.
Indeed, Table \ref{tab3bis} shows that, fixed $r=4.8$ and varying $n$, the number of iterations does not depend on $n$. The deteriorating behavior of GMRES when $n$ is fixed and $r$ increases (cf. Table \ref{tab3}) is due to the fact that some eigenvalues of $T_{200}(f_1)$ become close to zero, since the range of $\lambda_2(f_1(x))$ approaches zero as $r$ increases.
Looking at Figure \ref{casotot2}, if we use the preconditioner $T^{-1}_{200}(g_1)$, we improve the cluster of the eigenvalues and so the GMRES converges with a constant number of iterations (cf. Table \ref{tab3bis}), which is substantially independent both on $n$ and $r$.

This example fits with the theoretical results of this paper, since $f_1$ and $g_1$ are both bounded and $0\notin{\rm Coh}[\mathcal{ENR}(g_1)]$.
We stress that $g_1$ has been chosen so that the essential range of
$$h_1(x)=g_1^{-1}(x)f_1(x)=Q(x)\left(\begin{array}{cc}2+\mathbf i+\cos(x) & 0\\ 1/(5+r{\rm e}^{\mathbf i x}) & 1\end{array}\right)Q(x)^T,$$
which is given by
$$\mathcal{ER}(h_1)=\mathcal{ER}(\lambda_1(h_1))\bigcup\mathcal{ER}(\lambda_2(h_1))=\{t+\mathbf i:1\le t\le 3\}\bigcup\{1\},$$
is `compressed' and `well separated from 0' independently of the value of $r$.
In this way, since Theorem \ref{szego-nonherm-prec} and Remark \ref{Golinskii} ensure that the matrix-sequence $\{T_n^{-1}(g_1)T_n(f_1)\}$ is weakly clustered at $\mathcal{ER}(h_1)$, we expect a number of preconditioned GMRES iterations independent of $r,n$ and `small enough'. This is confirmed by the results in Tables \ref{tab3bis}--\ref{tab3}.
\end{Case}

\begin{Case}\label{Case 2}
Let us choose $A^{(2)}(x)$ and $B^{(2)}(x)$ as follows
\begin{equation*}
\begin{array}{cc}
A^{(2)}(x)=\left(
\begin{array}{cc}
2+{\mathbf i}+\cos(x) & 0 \\
1/(x^2-1) & 5+r{\rm e}^{{\mathbf i}x}
\end{array}
\right) , & B^{(2)}(x)=B^{(1)}(x)
\end{array}
\end{equation*}
and define
\begin{align*}
f_{2}(x) &=Q(x)A^{(2)}(x)Q(x)^{T} \\
g_{2}(x) &=Q(x)B^{(2)}(x)Q(x)^{T}.
\end{align*}
\begin{table}[htb]
\centering
\begin{minipage}[l]{.30\textwidth}
\centering
\begin{tabular}{c|cc}
& \multicolumn{2}{|c}{Iterations} \\
$n$ & No Prec. & Prec. \\ \hline
50 & 55 & 14 \\
100 & 98 & 14 \\
200 & 179 & 13 \\
400 & 230 & 14 \\
800 & 235 & 13
\end{tabular}
\caption{Number of GMRES iterations for $T_{n}(f_{1})$ and $T_{n}^{-1}(g_{1})T_{n}(f_{1})$ fixed $r=4.8$ and varying $n$}
\label{tab3bis}
\end{minipage}
\hspace{5mm}
\begin{minipage}[l]{.30\textwidth}
\centering
\begin{tabular}{c|cc}
& \multicolumn{2}{|c}{Iterations} \\
$r$ & No Prec. & Prec. \\ \hline
1 & 17 & 14 \\
2 & 22 & 14 \\
3 & 31 & 14 \\
4 & 55 & 14 \\
4.8 & 185 & 14
\end{tabular}
\caption{Number of GMRES iterations for $T_{200}(f_{1})$ and $T_{200}^{-1}(g_{1})T_{200}(f_{1})$ varying $r$}
\label{tab3}
\end{minipage}
\hspace{5mm}
\begin{minipage}[l]{.30\textwidth}
\centering
\begin{tabular}{c|cc}
& \multicolumn{2}{|c}{Iterations} \\
$r$ & No Prec. & Prec. \\ \hline
1 & 17 & 13 \\
2 & 22 & 13 \\
3 & 30 & 13 \\
4 & 53 & 13 \\
4.8 & 183 & 13
\end{tabular}
\caption{Number of GMRES iterations for $T_{200}(f_{2})$ and $T_{200}^{-1}(g_{2})T_{200}(f_{2})$ varying $r$}
\label{tab4}
\end{minipage}
\end{table}
Although this case is not covered by the theory, since $f_2$ is not bounded, we find that the eigenvalues of $T_{200}^{-1}(g_{2})T_{200}(f_{2})$ are closely related to the eigenvalues of $g_{2}^{-1}f_{2}$. The graphs of such eigenvalues (not reported here) are very similar to those in Figure \ref{casotot2}. Table \ref{tab4} shows the number of GMRES iterations, setting $n=200$ and moving the radius of the disk used to define $f_{2}$ and $g_{2}$.
\end{Case}

\begin{figure}[htb]
\centering
\begin{minipage}[l]{.30\textwidth}
\includegraphics[width=5cm, height=4cm]{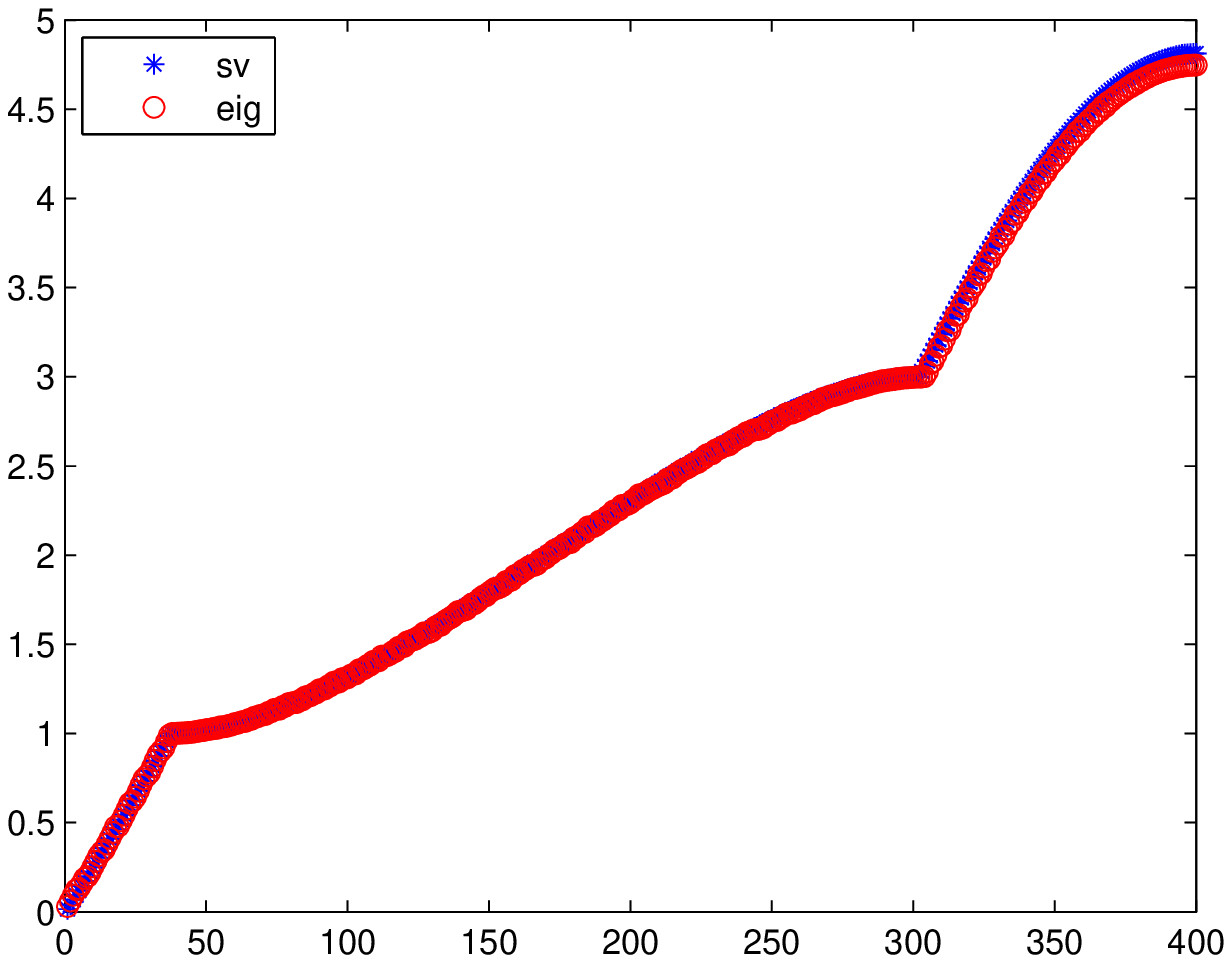}
\caption{Singular values and moduli of the eigenvalues of $T_{200}(f_3)$}\label{caso1svtilde}
\end{minipage}
\hspace{10mm}
\begin{minipage}[l]{.30\textwidth}
\includegraphics[width=5cm, height=4cm]{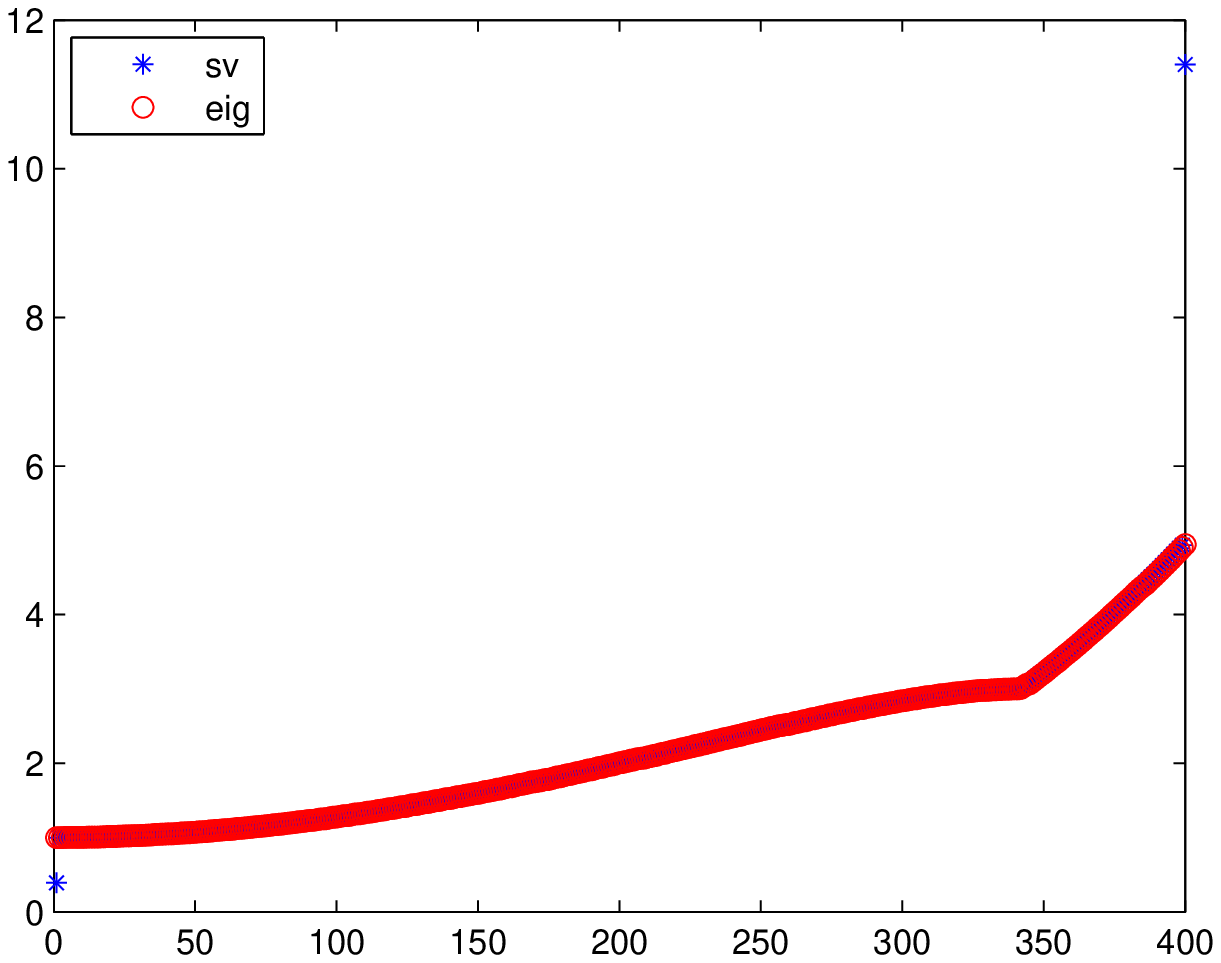}
\caption{Singular values and moduli of the eigenvalues of $T_{200}^{-1}(g_3)T_{200}(f_3)$}
\label{caso1svpretilde}
\end{minipage}
\\
\begin{minipage}[l]{.30\textwidth}
\includegraphics[width=5cm, height=4cm]{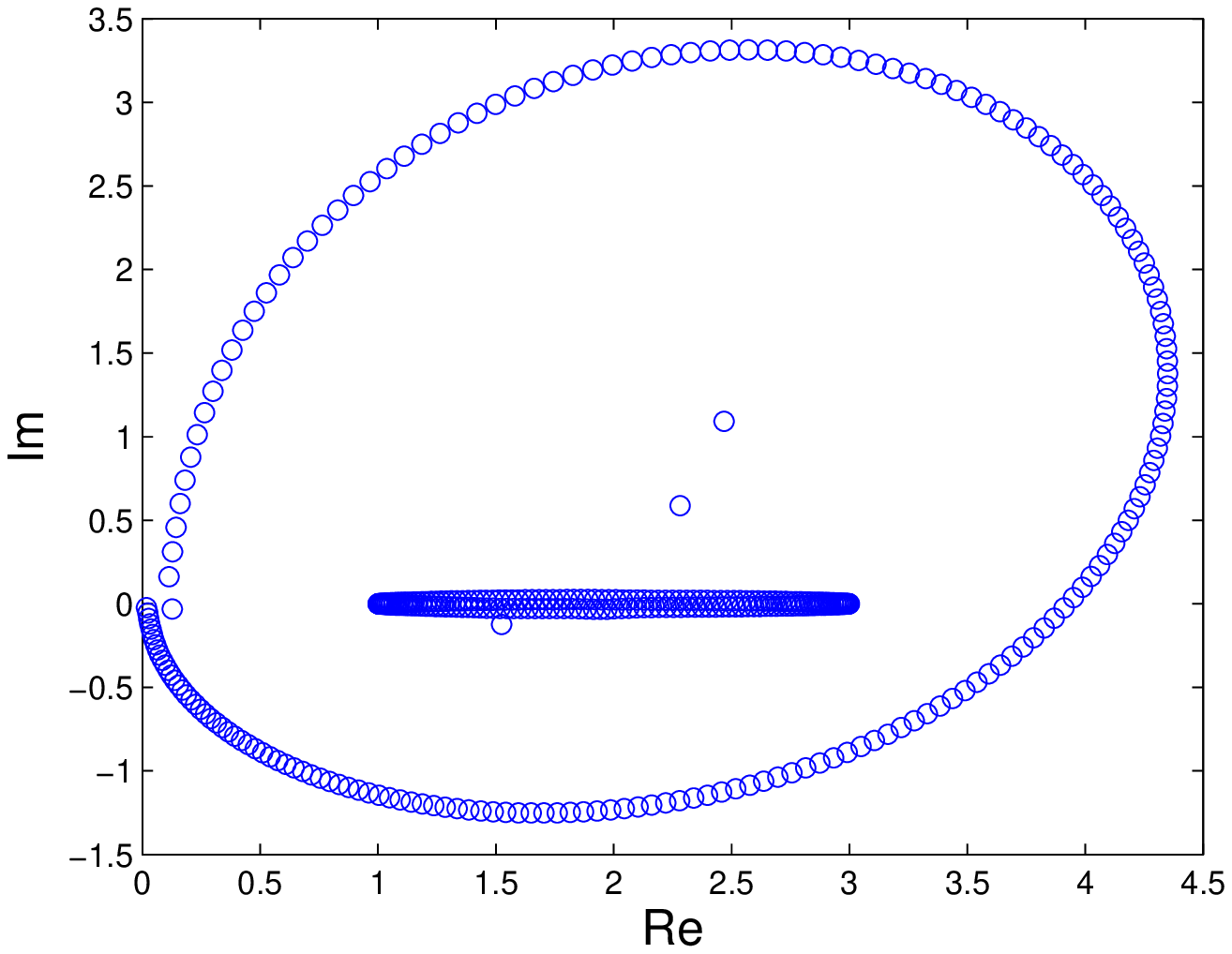}
\caption{Eigenvalues in the complex plane of $T_{200}(f_3)$}\label{caso1eigtilde}
\end{minipage}
\hspace{10mm}
\begin{minipage}[l]{.30\textwidth}
\includegraphics[width=5cm, height=4cm]{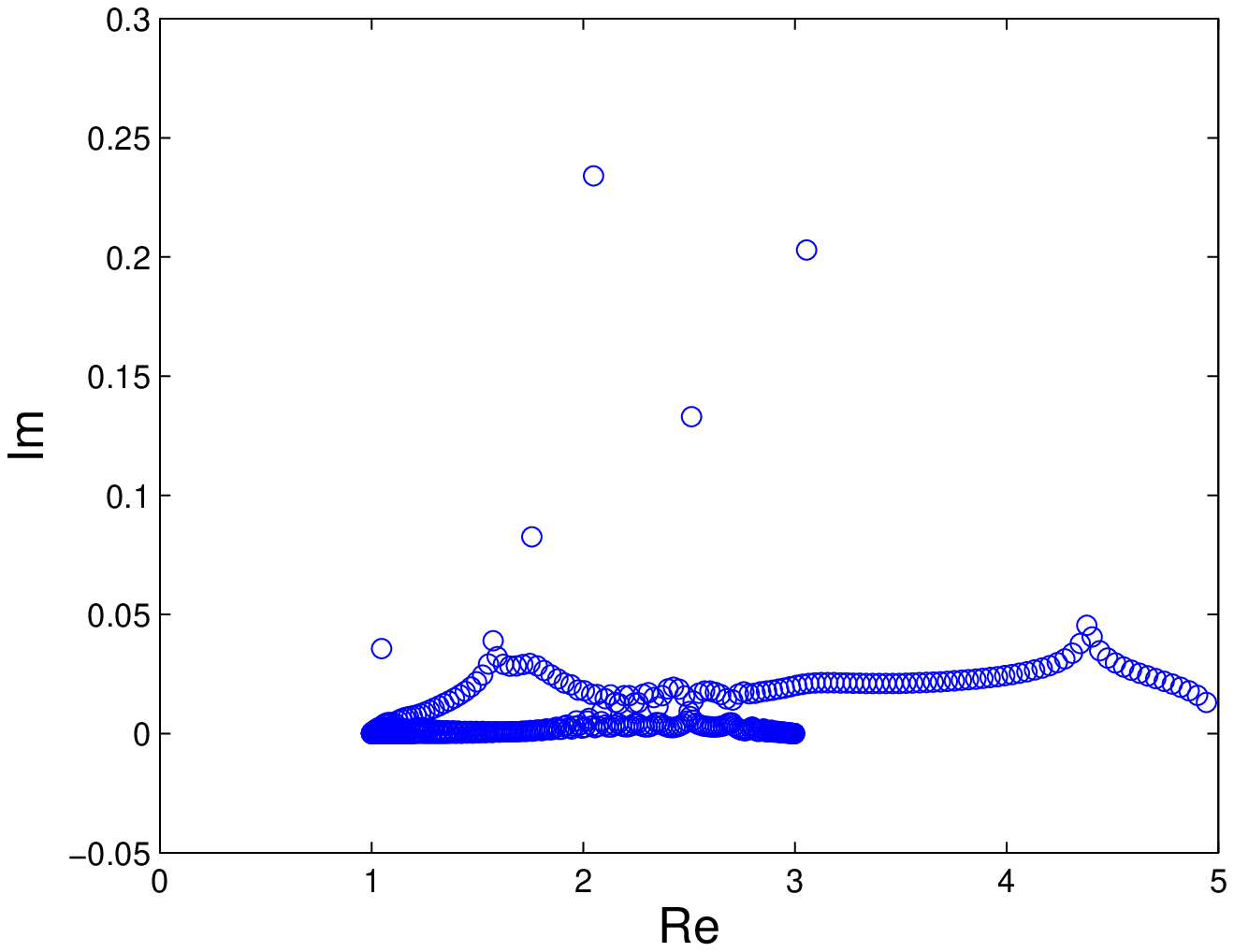}
\caption{Eigenvalues in the complex plane of $T_{200}^{-1}(g_3)T_{200}(f_3)$}\label{caso1eigpretilde}
\end{minipage}
\end{figure}

\begin{Case}\label{Case 3}
Let us choose $A^{(3)}(x)$ and $B^{(3)}(x)$ as follows
\begin{equation*}
\begin{array}{cc}
A^{(3)}(x)=\left(
\begin{array}{cc}
(1-{\rm e}^{{\mathbf i}x})\left(1+x^2/\pi^2\right)  & 0 \\
0 & 2+\cos (x)
\end{array}
\right) , & B^{(3)}(x)=\left(
\begin{array}{cc}
(1-{\rm e}^{{\mathbf i}x}) & 0 \\
0 & 1
\end{array}
\right)
\end{array}
\end{equation*}
and define
\begin{align*}
f_{3}(x) &=Q(x)A^{(3)}(x)Q(x)^{T} \\
g_{3}(x) &=Q(x)B^{(3)}(x)Q(x)^{T}.
\end{align*}
\begin{figure}[htb]
\centering
\begin{minipage}[l]{.30\textwidth}
\includegraphics[width=5cm, height=4cm]{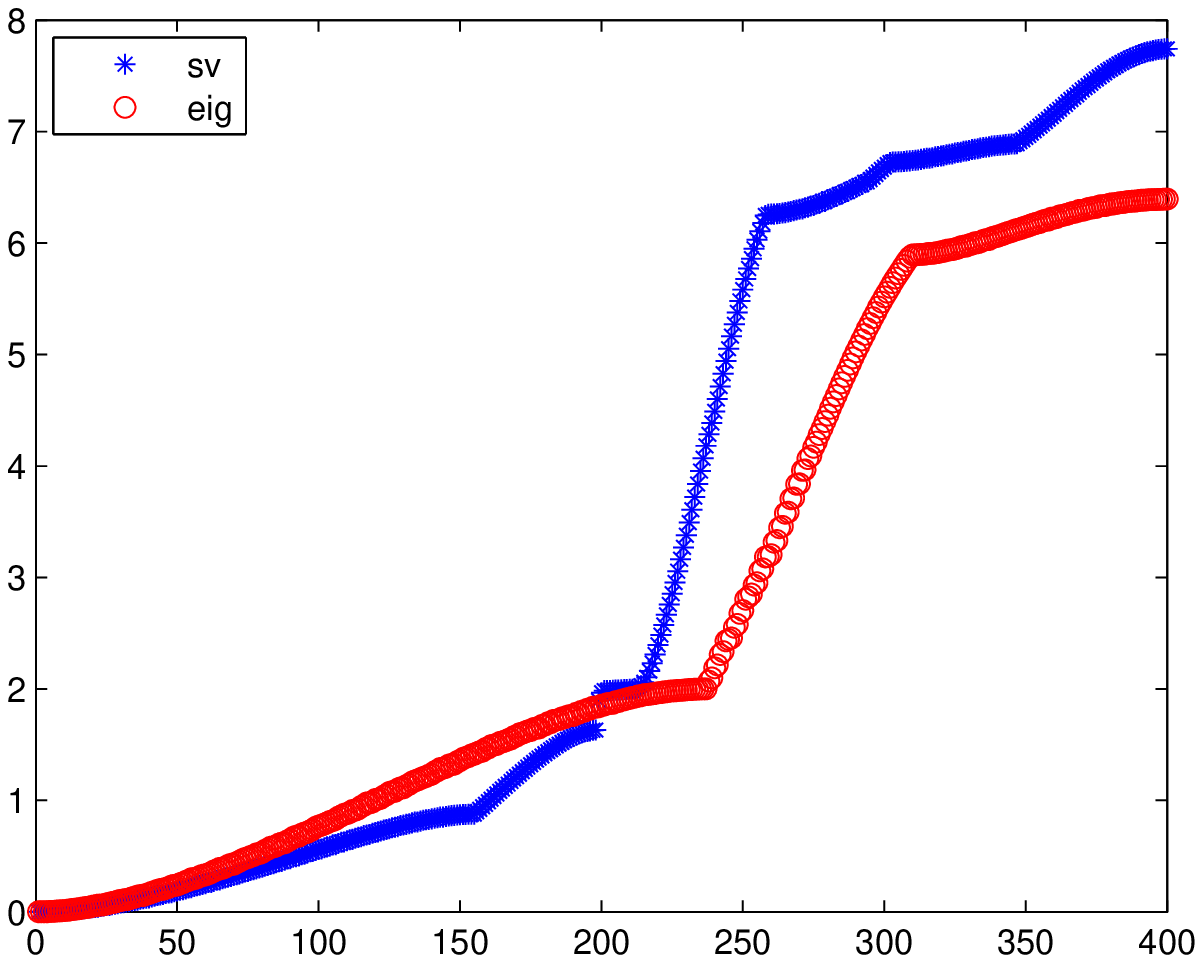}
\caption{Singular values and moduli of the eigenvalues
of $T_{200}(f_4)$}\label{caso2sv}
\end{minipage}
\hspace{10mm}
\begin{minipage}[l]{.30\textwidth}
\includegraphics[width=5cm, height=4cm]{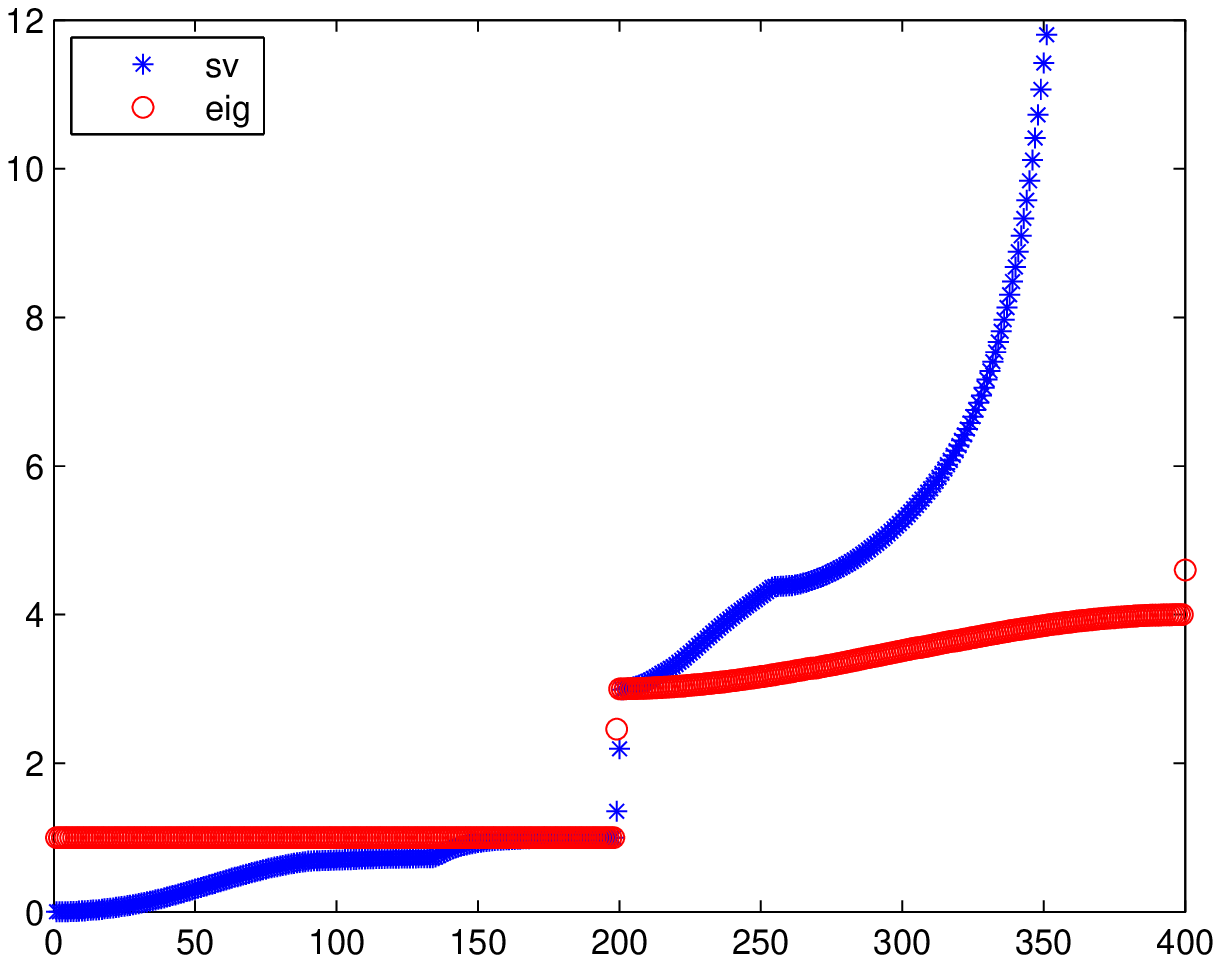}
\caption{Singular values and moduli of the eigenvalues of $T_{200}^{-1}(g_4)T_{200}(f_4)$}\label{caso2svpre}
\end{minipage}
\\
\begin{minipage}[l]{.30\textwidth}
\includegraphics[width=5cm, height=4cm]{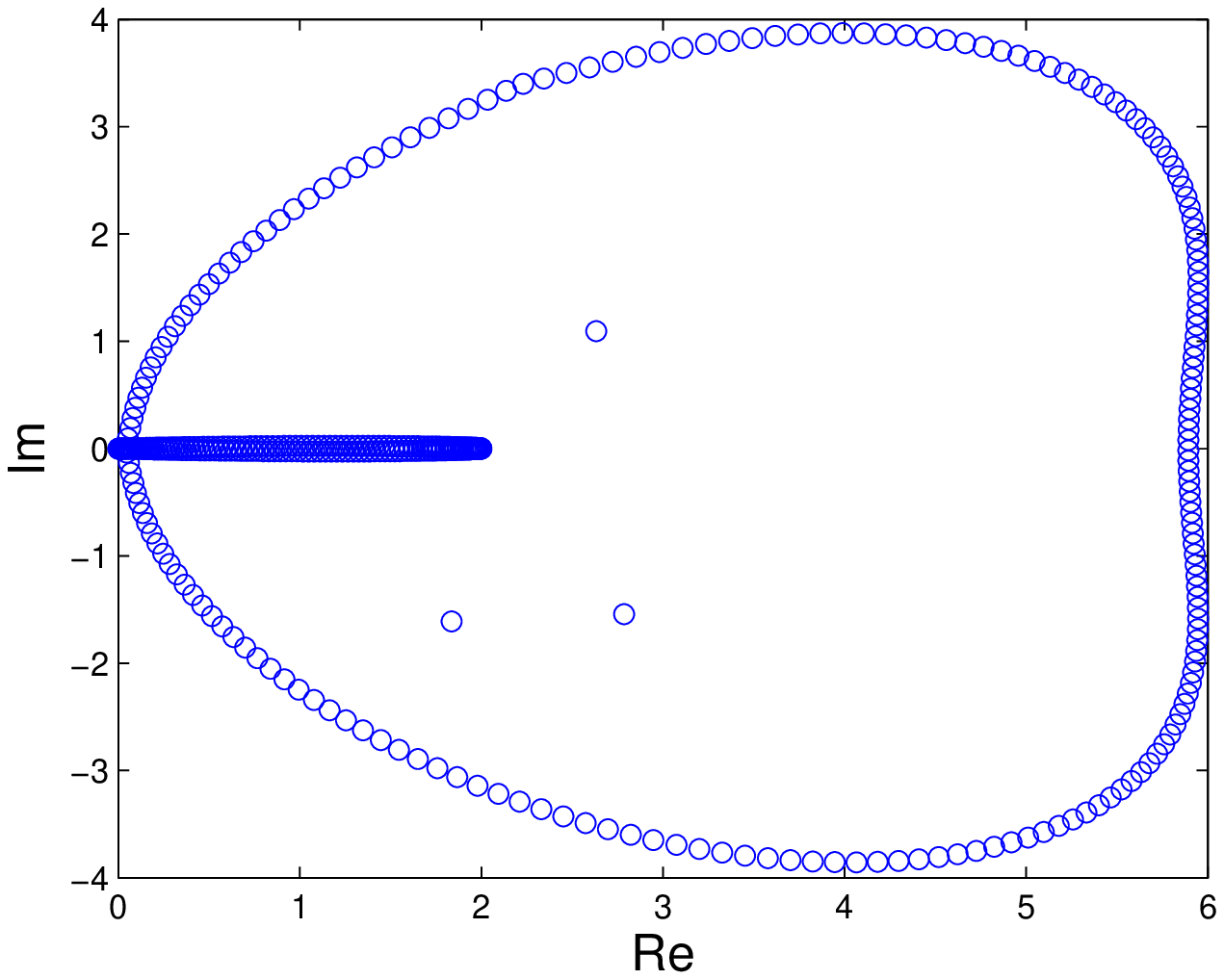}
\caption{Eigenvalues in the complex plane
of $T_{200}(f_4)$}\label{caso2eig}
\end{minipage}
\hspace{10mm}
\begin{minipage}[l]{.30\textwidth}
\includegraphics[width=5cm, height=4cm]{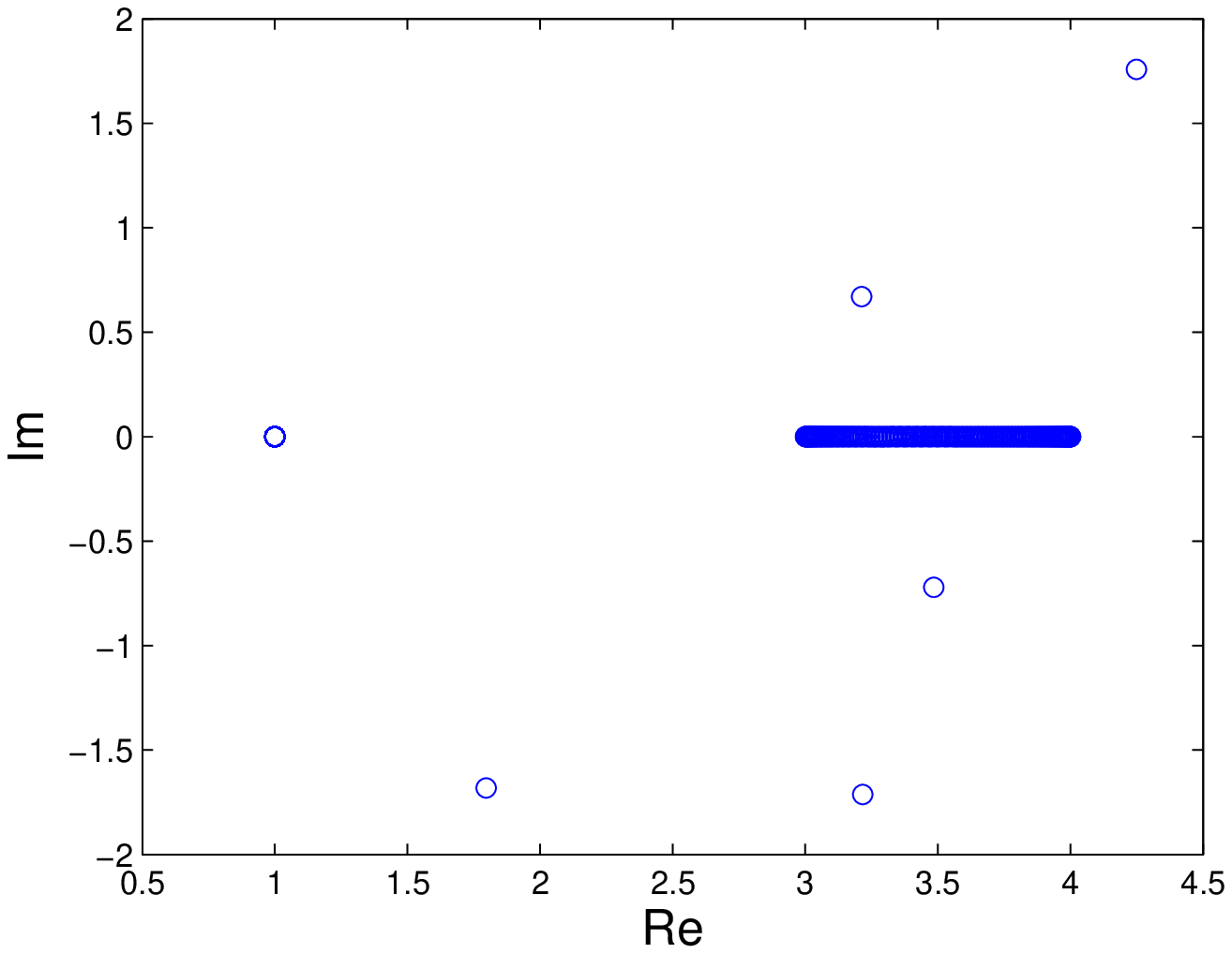}
\caption{Eigenvalues in the complex plane of $T_{200}^{-1}(g_4)T_{200}(f_4)$}\label{caso2eigpre}
\end{minipage}
\end{figure}

This case can be obtained from Example \ref{Esempio} selecting $\varphi_1(x)=1+x^2/\pi^2$ and $\varphi_2(x)=2+\cos(x)$.
Figures \ref{caso1svtilde} and \ref{caso1svpretilde} show the singular values and the moduli of the eigenvalues of $T_{200}(f_3)$ and $T_{200}^{-1}(g_3)T_{200}(f_3)$, respectively. Let us observe that the singular values and the moduli of the eigenvalues are, for both matrices, almost superposed. Figures \ref{caso1eigtilde} and \ref{caso1eigpretilde} refer to the eigenvalues in the complex plane of the same
matrices. As already argued for Case \ref{Case 1}, even in this case the eigenvalues of $T_{200}(f_3)$ show two different behaviors: a half of the eigenvalues is clustered at $[1,3]$, which is the range of the function $\lambda_2(f_3(x))=A_{2,2}^{(3)}(x)$, the others mimic $\lambda_1(f_3(x))=A_{1,1}^{(3)}(x)$, drawing a circle passing near to $0$. The closeness of the eigenvalues to $0$ is responsible of the non-optimality of the GMRES method when we solve a linear system with matrix $T_{200}(f_3)$.
\begin{table}[htb]
\centering
\begin{minipage}[l]{.30\textwidth}
\centering
\begin{tabular}{c|cc}
& \multicolumn{2}{|c}{Iterations} \\
$n$ & No Prec. & Prec. \\ \hline
50 & 59 & 15 \\
100 & 106 & 16 \\
200 & 192 & 16 \\
400 & 343 & 16
\end{tabular}
\caption{Number of GMRES iterations for $T_{n}(f_{3})$ and $T^{-1}_{n}(g_3)T_{n}(f_{3})$ varying $n$}
\label{tab1}
\end{minipage}
\hspace{5mm}
\begin{minipage}[l]{.30\textwidth}
\centering
\begin{tabular}{c|cc}
& \multicolumn{2}{|c}{Iterations} \\
$n$ & No Prec. & Prec. \\ \hline
50 & 100 & 9 \\
100 & 200 & 9 \\
200 & 338 & 9 \\
400 & 577 & 9
\end{tabular}
\caption{Number of GMRES iterations for $T_{n}(f_{4})$ and $T_{n}^{-1}(g_{4})T_{n}(f_{4})$ varying $n$}
\label{tab2}
\end{minipage}
\end{table}
Indeed, as can be observed in Table \ref{tab1}, the number of GMRES iterations required to reach tolerance $10^{-6}$ increases with $n$ for $T_{n}(f_{3})$. The preconditioned matrix $T^{-1}_{n}(g_3)T_n(f_3)$ has eigenvalues far from $0$ and bounded in modulus (see Figure \ref{caso1eigpretilde}) and so the preconditioned GMRES converges with a constant number of iterations (see Table \ref{tab1}).
This example is not covered by the theory explained in previous sections, since Coh$[{\cal ENR}(g_3)]$ includes the complex zero. The numerical tests, however, show that there is room for improving the theory, by allowing the symbol of the preconditioner to have eigenvalues assuming zero value.
\end{Case}

\begin{Case}\label{Case 4}
Let us choose $A^{(4)}(x)$ and $B^{(4)}(x)$ as follows
\begin{equation*}
\begin{array}{cc}
A^{(4)}(x)=\left(
\begin{array}{cc}
(1-{\rm e}^{{\mathbf i}x})\left( \sin ^{2}(x)+3\right)  & 0 \\
x & 1+\cos (x)
\end{array}
\right) , & B^{(4)}(x)=\left(
\begin{array}{cc}
(1-{\rm e}^{{\mathbf i}x})  & 0 \\
0 & 1+\cos (x)
\end{array}
\right)
\end{array}
\end{equation*}
and define
\begin{align*}
f_{4}(x) &=Q(x)A^{(4)}(x)Q(x)^{T} \\
g_{4}(x) &=Q(x)B^{(4)}(x)Q(x)^{T}.
\end{align*}
Figures \ref{caso2sv} and \ref{caso2svpre} show the singular values and the
moduli of the eigenvalues of $T_{200}(f_{4})$ and $T_{200}^{-1}(g_{4})T_{200}(f_{4})$, respectively. For a better resolution, in Figure \ref{caso2svpre} some singular values of order about $10^{-4}$ have been cut. Figures \ref{caso2eig} and
\ref{caso2eigpre} refer to the eigenvalues in the complex plane of the same
matrices. The reasoning regarding the behavior of the eigenvalues applies as in Case \ref{Case 3}.
Table \ref{tab2} shows the number of GMRES iterations required to reach the prescribed tolerance varying $n$. Again the number of iterations increases with $n$ for $T_{n}(f_{4})$, while in the preconditioned case the related iteration count
remains constant. For the same reason of the previous example, even in this case Theorem \ref{szego-nonherm-prec} does not apply, but again the numerical results give us hope for improving our tools.
\end{Case}
\begin{figure}[htb]
\centering
\begin{minipage}[l]{.30\textwidth}
\includegraphics[width=5cm, height=4cm]{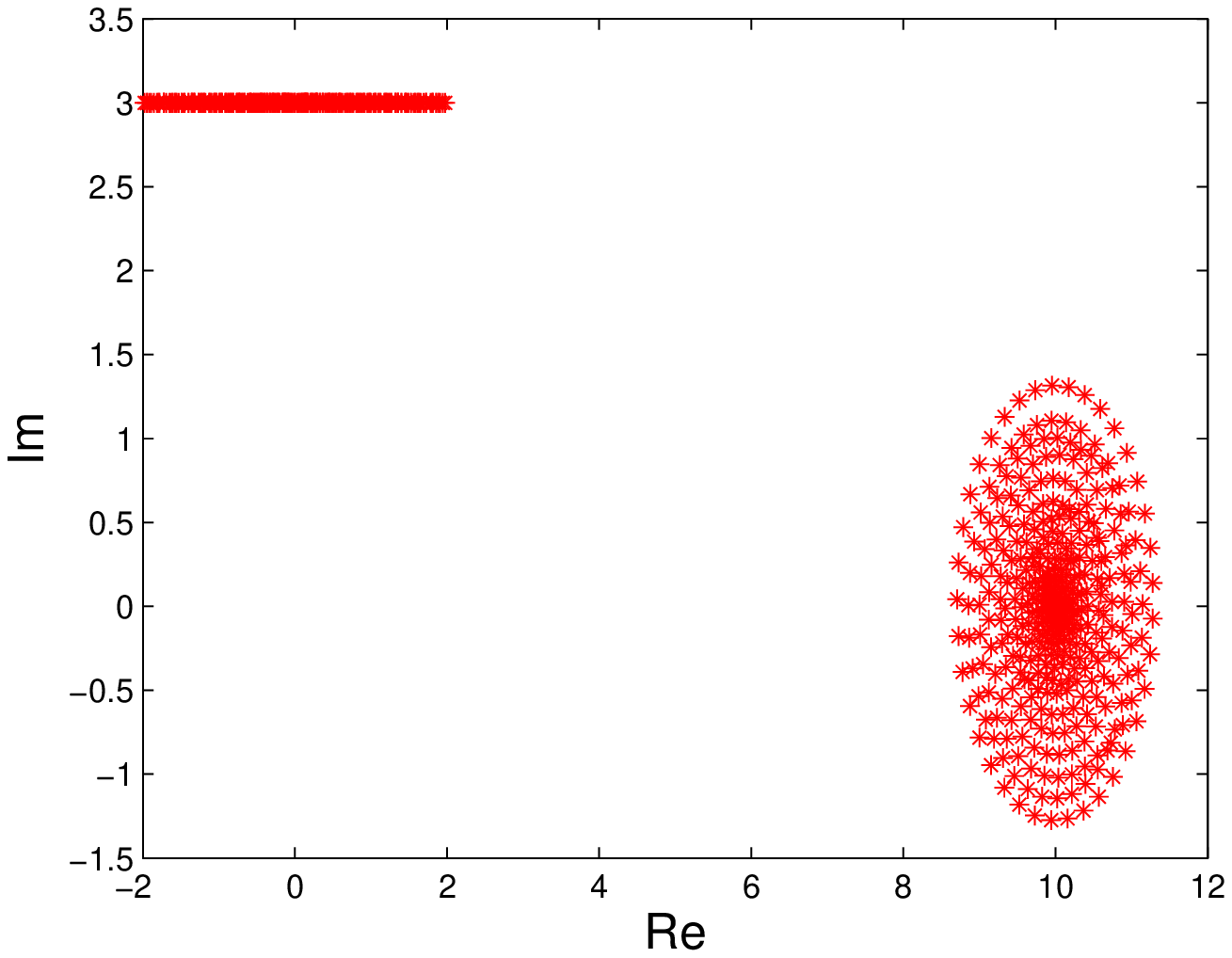}
\caption{Eigenvalues in the complex plane of $T_{n}(A^{(5)})$ with $n=(20,20)$}\label{eigA5}
\end{minipage}
\hspace{10mm}
\begin{minipage}[l]{.30\textwidth}
\includegraphics[width=5cm, height=4cm]{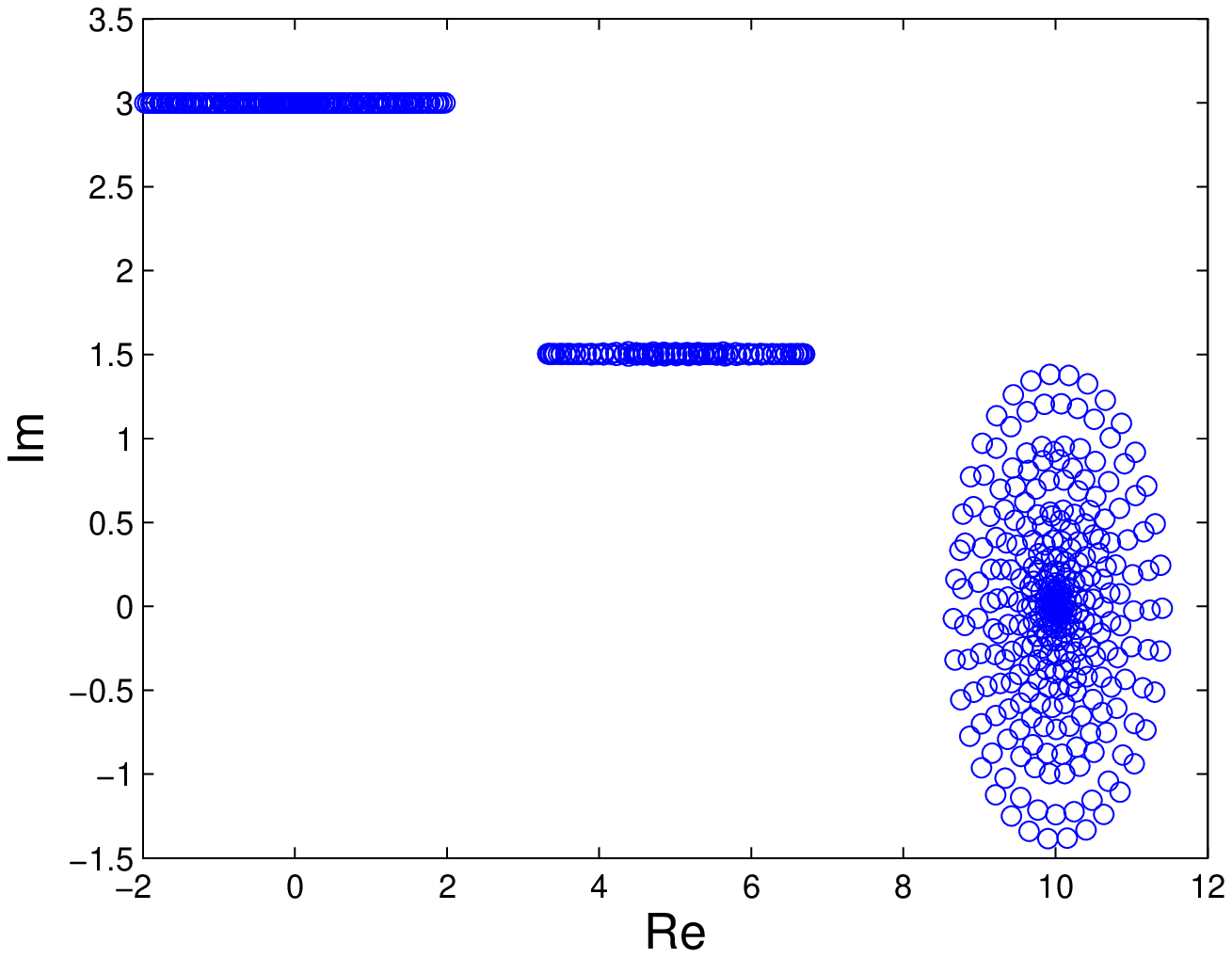}
\caption{Eigenvalues in the complex plane of $T_{n}(f_5)$ with $n=(20,20)$}\label{eigf5}
\end{minipage}
\end{figure}
\begin{figure}[htb]
\centering
\begin{minipage}[l]{.30\textwidth}
\includegraphics[width=5cm, height=4cm]{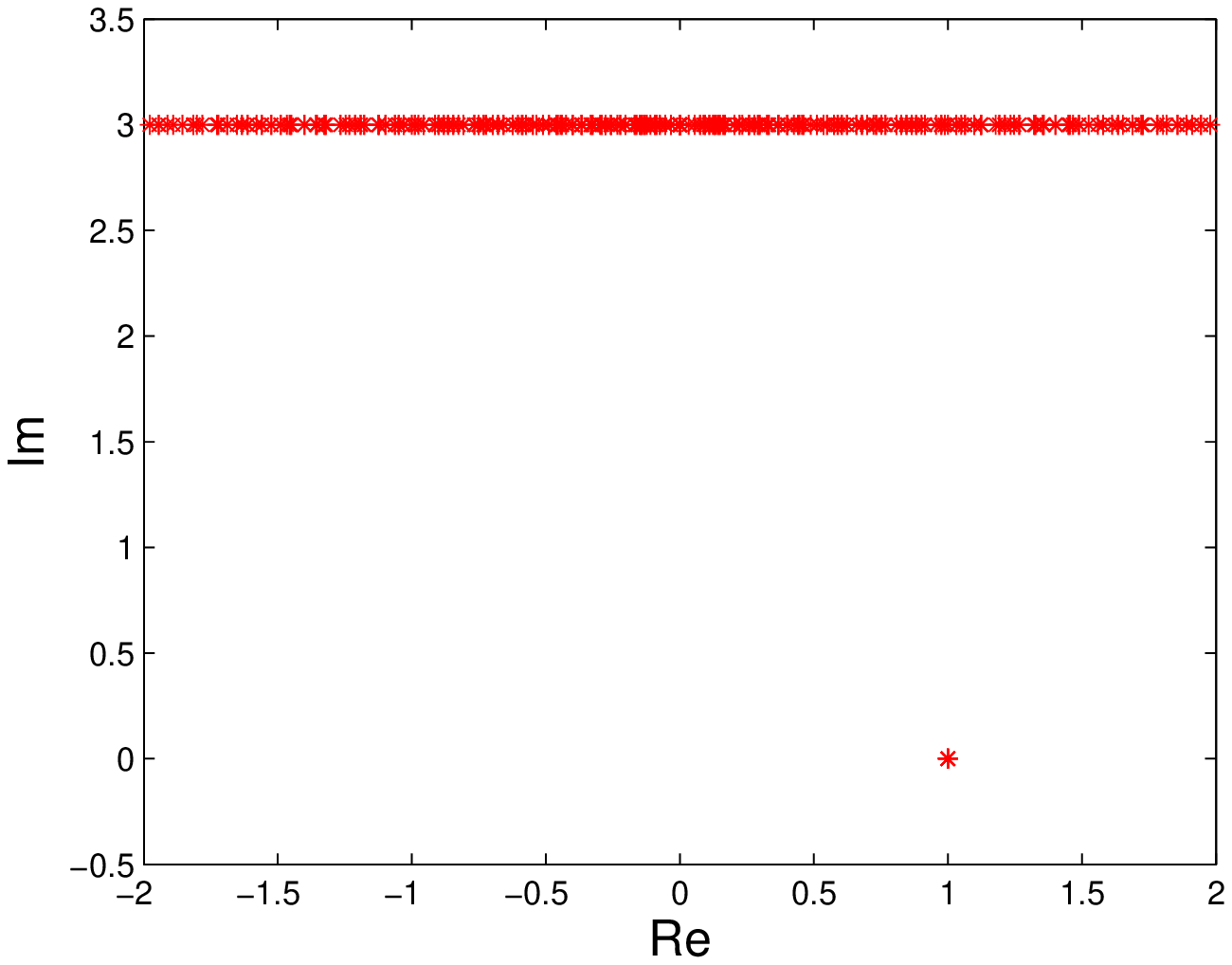}
\caption{Eigenvalues in the complex plane of $T^{-1}_{n}(B^{(5)})T_{n}(A^{(5)})$ with $n=(20,20)$}\label{eigA5pre}
\end{minipage}
\hspace{10mm}
\begin{minipage}[l]{.30\textwidth}
\includegraphics[width=5cm, height=4cm]{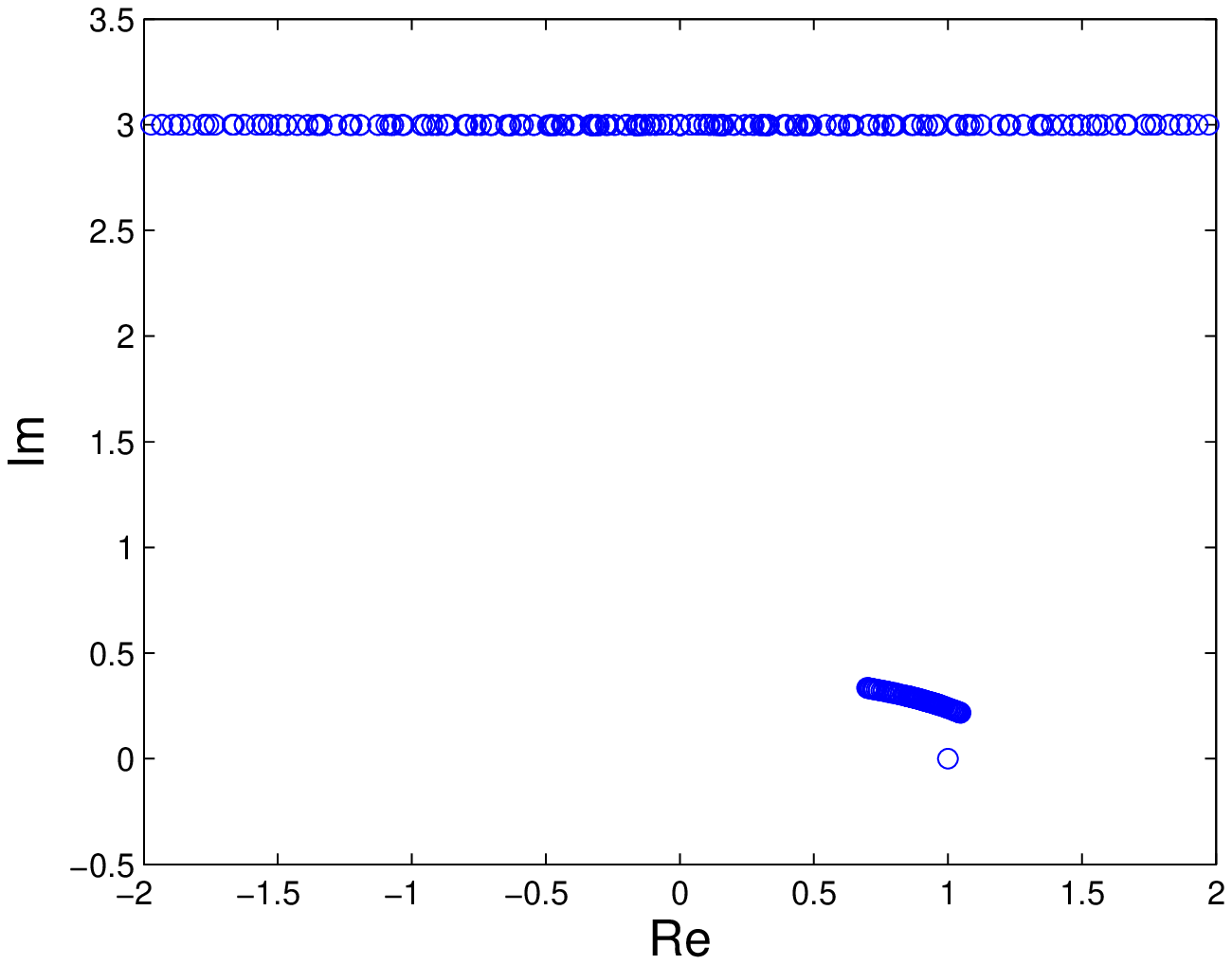}
\caption{Eigenvalues in the complex plane of $T^{-1}_{n}(g_5)T_{n}(f_5)$ with $n=(20,20)$}\label{eigf5pre}
\end{minipage}
\end{figure}
\begin{table}[htb]
\centering
\begin{tabular}{cc|cc}
& &\multicolumn{2}{|c}{Counting the outliers} \\
$n_1$ & $n_2$ & Out. & Out./$\sqrt{\widehat n}$ \\ \hline
5 & 5 & 32 & 6.40\\
10 & 10 & 72 & 7.20\\
15 & 15 & 112 & 7.47\\
20 & 20 & 152 & 7.60\\
25 & 25 & 192 & 7.68\\
30 & 30 & 232 & 7.73
\end{tabular}
\caption{Number of outliers for both $T_{n}(f_{5})$ and $T^{-1}_{n}(g_5)T_{n}(f_{5})$ varying $n_1$ and $n_2$
}
\label{outliers}
\end{table}
\begin{figure}[htb]
\centering
\begin{minipage}[l]{.30\textwidth}
\includegraphics[width=5cm, height=4cm]{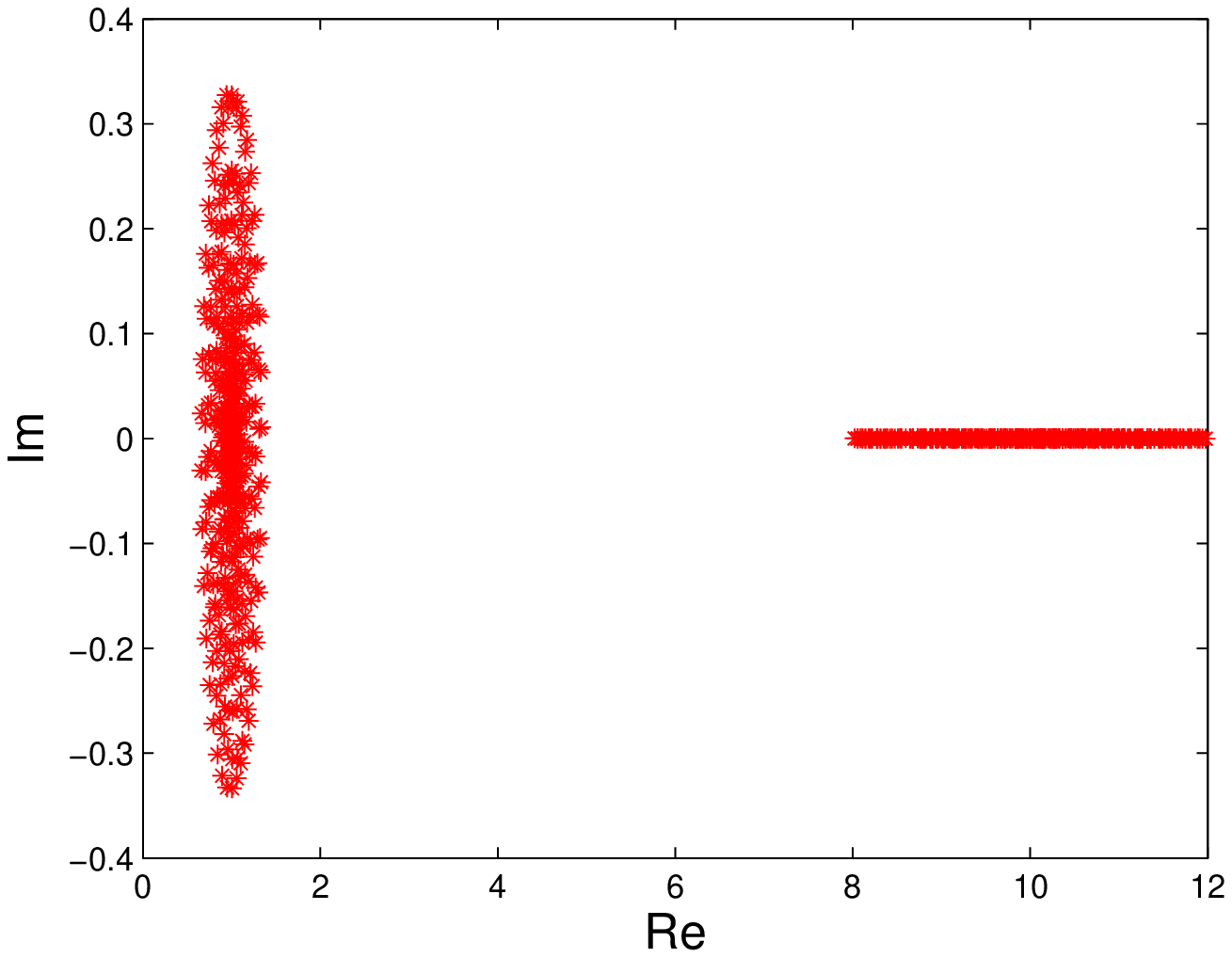}
\caption{Eigenvalues in the complex plane of $T_{n}(A^{(6)})$ with $n=(20,20)$}\label{eigA6}
\end{minipage}
\hspace{10mm}
\begin{minipage}[l]{.30\textwidth}
\includegraphics[width=5cm, height=4cm]{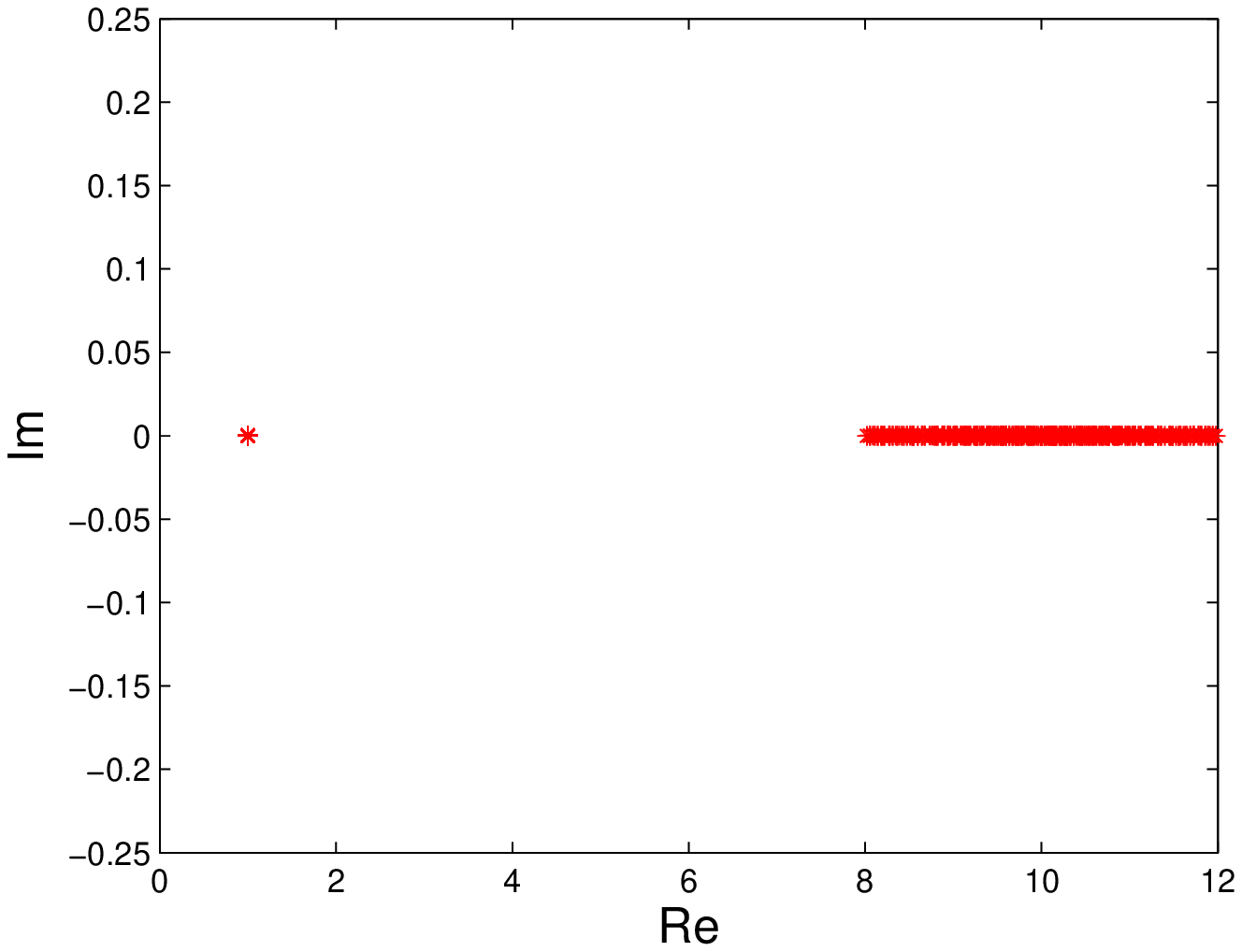}
\caption{Eigenvalues in the complex plane of $T^{-1}_{n}(B^{(6)})T_{n}(A^{(6)})$ with $n=(20,20)$}\label{eigA6pre}
\end{minipage}
\end{figure}

\subsection{2-level examples}

In this section we fix $s=2$ and $k=2$, that is we consider ${\cal M}_2$-valued symbols of $2$ variables. In particular, we extend the definitions of $f$ and $g$ given in (\ref{fg}) taking $x=(x_1,x_2)$ and
\begin{equation*}
Q(x)=\left(
\begin{array}{cc}
\cos (x_1+x_2) & \sin(x_1+x_2) \\
-\sin (x_1+x_2) & \cos(x_1+x_2)
\end{array}
\right).
\end{equation*}
From here onwards, $n$ is a $2$-index, that is of type $n=(n_1,n_2)$.

\begin{Case}\label{Case 5}
This case can be seen as a 2-level extension of Case \ref{Case 1} obtained by choosing
\begin{equation*}
\begin{array}{cc}
A^{(5)}(x)=\left(
\begin{array}{cc}
3{\mathbf i}+\cos(x_1)+\cos(x_2) & 0 \\
0 &10+2({\rm e}^{{\mathbf i}x_1}+{\rm e}^{{\mathbf i}x_2})
\end{array}
\right) , & B^{(5)}(x)=\left(
\begin{array}{cc}
1 & 0 \\
0 & 10+2({\rm e}^{{\mathbf i}x_1}+{\rm e}^{{\mathbf i}x_2})
\end{array}
\right)
\end{array}
\end{equation*}
and defining
\begin{align*}
f_{5}(x) &=Q(x)A^{(5)}(x)Q(x)^{T} \\
g_{5}(x) &=Q(x)B^{(5)}(x)Q(x)^{T}.
\end{align*}
Let us observe that $f_5$ and $A^{(5)}$ are similar via the unitary transformation $Q(x)$, then, according to the theory, the associated 2-level block Toeplitz matrices are distributed in the sense of the eigenvalues in the same way. As shown in Figures \ref{eigA5} and \ref{eigf5}, in which $n=(n_1,n_2)=(20,20)$, both the eigenvalues of $T_{n}(f_{5})$ and $T_{n}(A^{(5)})$ are divided in two sub-clusters, one at the range of $A^{(5)}_{1,1}(x)=\lambda_1(f_5(x))$, the other at the range of $A^{(5)}_{2,2}(x)=\lambda_2(f_5(x))$. Interestingly enough, for $T_{n}(A^{(5)})$ the clusters are of strong type, while in the case of $T_{n}(f_{5})$ the spectrum presents outliers with real part in $(3,7)$ and imaginary part equal to $1.5$. Table \ref{outliers} shows that the number of outliers seems to behave as $o(\widehat n)$ or, more specifically, as $O(\sqrt{ \widehat n})$ (notice that this estimate is in line with the analysis in \cite{tyrty1}). Analogous results are obtained in the comparison between $T^{-1}_{n}(g_{5})T_{n}(f_{5})$ and $T^{-1}_{n}(B^{(5)})T_{n}(A^{(5)})$, as shown in Figures \ref{eigA5pre} and \ref{eigf5pre}. Refer again to Table \ref{outliers} for the number of outliers of $T_n^{-1}(g_5)T_n(f_5)$ varying $n_1$ and $n_2$ (this number is exactly the same as the number of outliers of $T_n(f_5)$).
The eigenvalues of the symbol $f_5$ have no zeros, so the number of GMRES iterations required to reach tolerance $10^{-6}$ in solving the system associated to $T_n(f_5)$ is optimal, that is it does not depend on $n$. However, as shown in Table \ref{tab5}, when preconditioning with $T_n(g_5)$, we preserve the optimality with a smaller number of iterations.
\end{Case}
\begin{figure}[htb]
\centering
\begin{minipage}[l]{.30\textwidth}
\includegraphics[width=5cm, height=4cm]{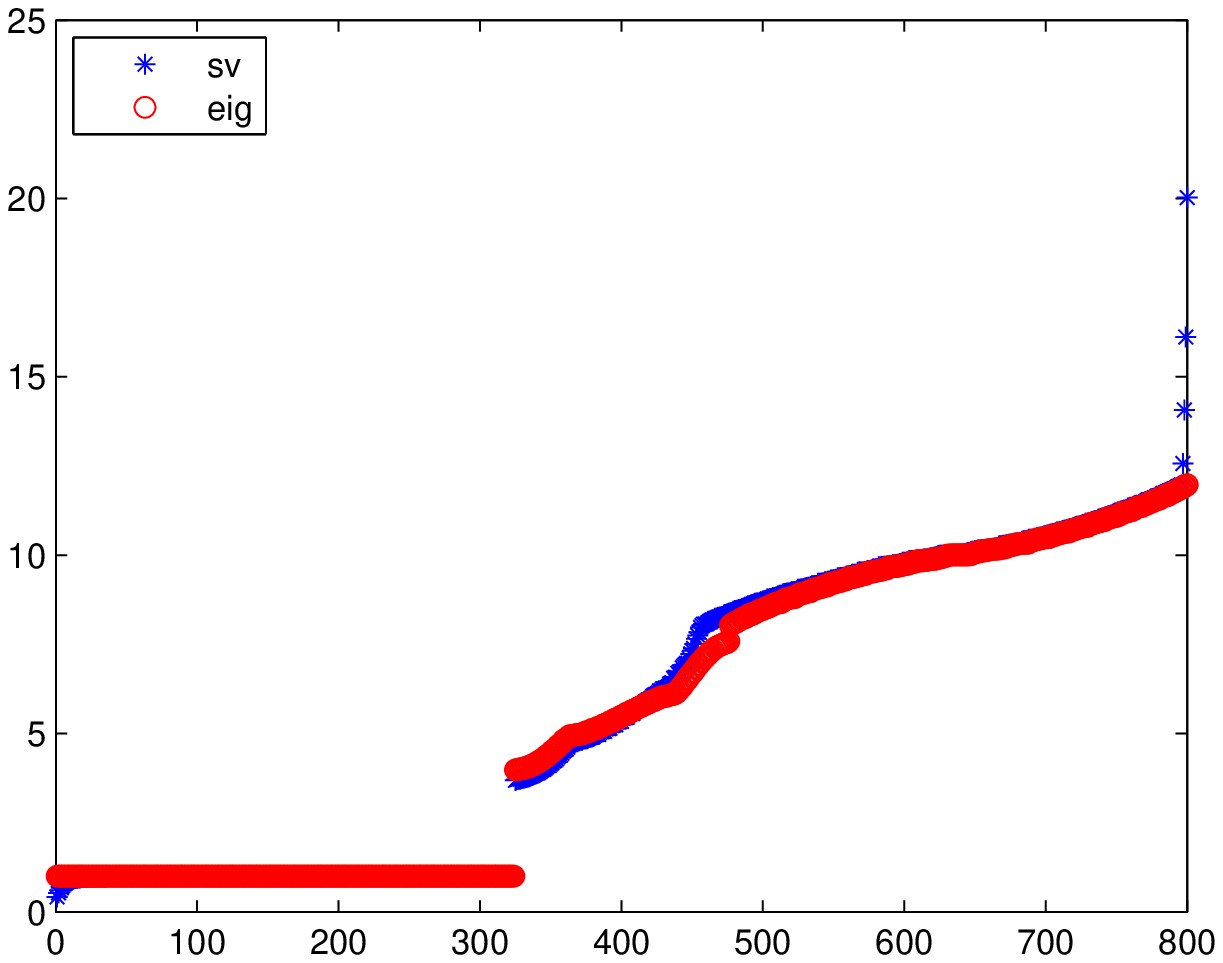}
\caption{Singular values and moduli of the eigenvalues of $T_{n}(f_6)$ with $n=(20,20)$}\label{caso6sv}
\end{minipage}
\hspace{10mm}
\begin{minipage}[l]{.30\textwidth}
\includegraphics[width=5cm, height=4cm]{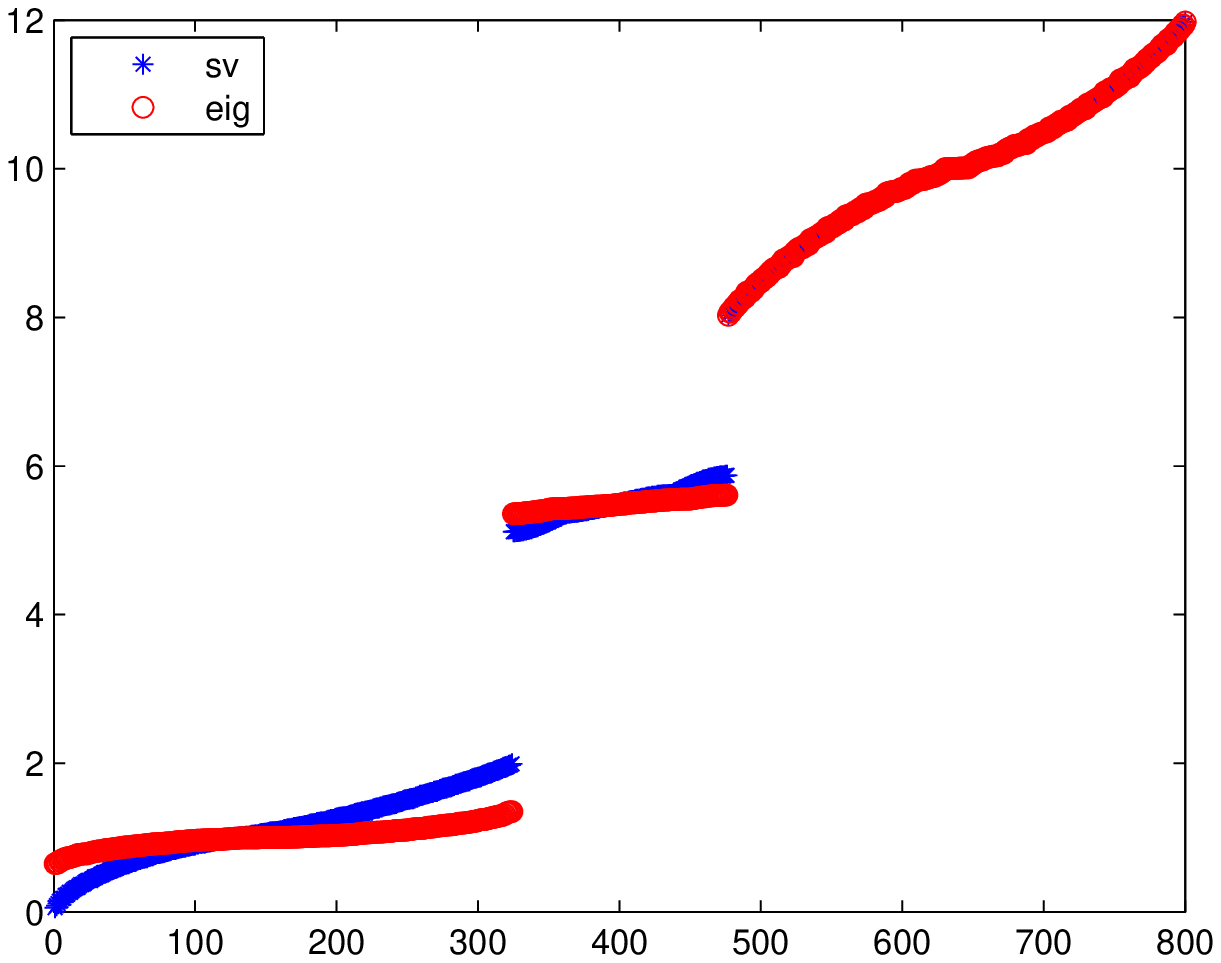}
\caption{Singular values and moduli of the eigenvalues of $T_{n}^{-1}(g_6)T_{n}(f_6)$ with $n=(20,20)$}\label{caso6svpre}
\end{minipage}
\\
\begin{minipage}[l]{.30\textwidth}
\includegraphics[width=5cm, height=4cm]{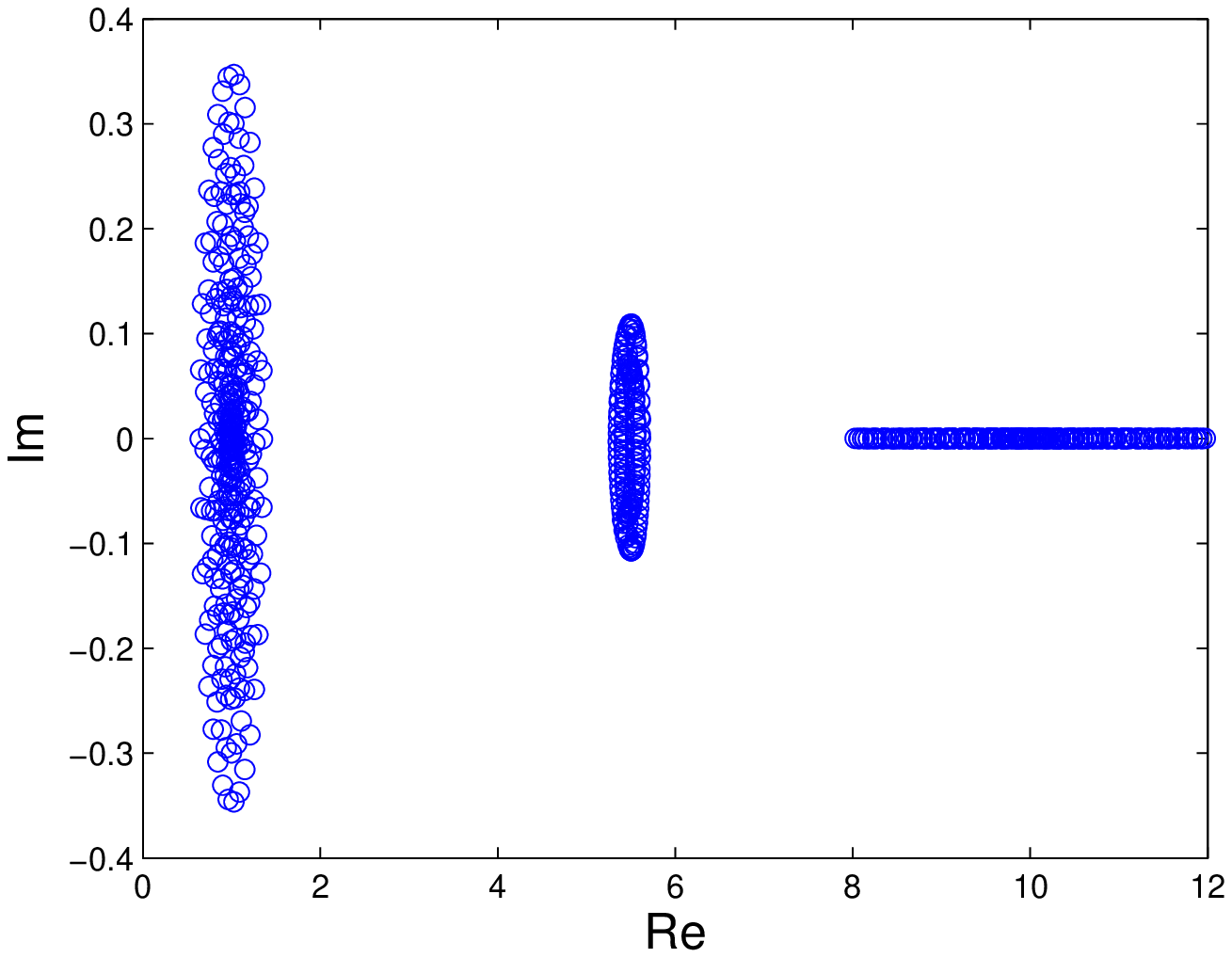}
\caption{Eigenvalues in the complex plane of $T_{n}(f_6)$ with $n=(20,20)$}\label{caso6eig}
\end{minipage}
\hspace{10mm}
\begin{minipage}[l]{.30\textwidth}
\includegraphics[width=5cm, height=4cm]{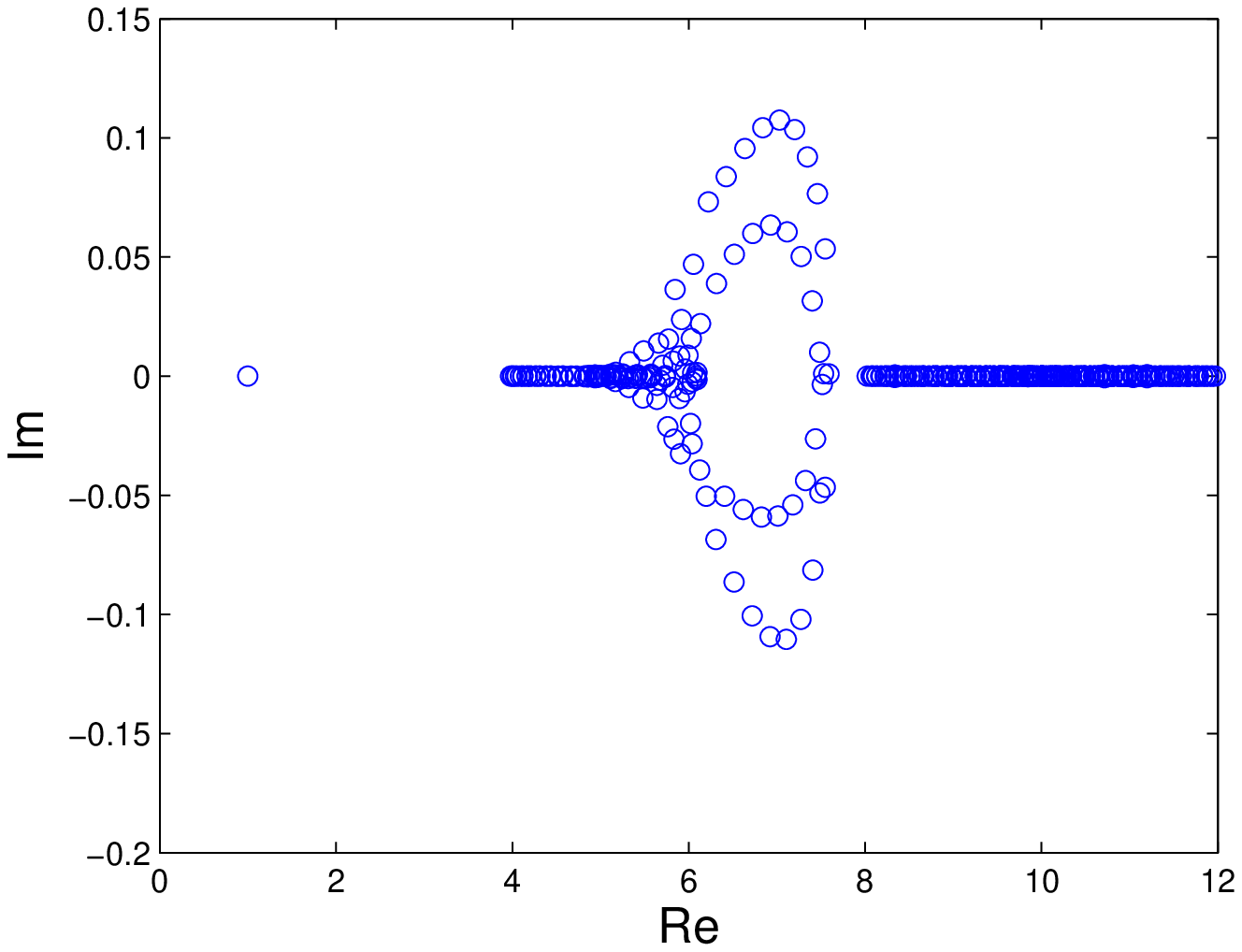}
\caption{Eigenvalues in the complex plane of $T_{n}^{-1}(g_6)T_{n}(f_6)$ with $n=(20,20)$}\label{caso6eigpre}
\end{minipage}
\end{figure}
\begin{table}[htb]
\centering
\begin{minipage}[l]{.30\textwidth}
\centering
\begin{tabular}{cc|cc}
& &\multicolumn{2}{|c}{Iterations} \\
$n_1$ & $n_2$ & No Prec. & Prec. \\ \hline
5 & 5 & 21 & 15 \\
10 & 10 & 33 & 24 \\
15 & 15 & 39 & 24 \\
20 & 20 & 42 & 25
\end{tabular}
\caption{Number of GMRES iterations for $T_{n}(f_{5})$ and $T_{n}^{-1}(g_{5})T_{n}(f_{5})$ varying $n_1$ and $n_2$}
\label{tab5}
\end{minipage}
\hspace{5mm}
\begin{minipage}[l]{.30\textwidth}
\centering
\begin{tabular}{cc|cc}
& & \multicolumn{2}{|c}{Iterations} \\
$n_1$ & $n_2$ & No Prec. & Prec. \\ \hline
5 & 5 & 17 & 14 \\
10 & 10 &39 & 16 \\
15 & 15 & 61 & 15 \\
20 & 20 & 80 & 16
\end{tabular}
\caption{Number of GMRES iterations for $T_{n}(f_{6})$ and $T_{n}^{-1}(g_{6})T_{n}(f_{6})$ varying $n_1$ and $n_2$}\label{tab6}
\end{minipage}
\end{table}

\begin{Case}\label{Case 6}
Let us choose $A^{(6)}$ and $B^{(6)}$ as
\begin{equation*}
\begin{array}{cc}
A^{(6)}(x)=\left(
\begin{array}{cc}
1-({\rm e}^{{\mathbf i}x_1}+{\rm e}^{{\mathbf i}x_2})/2 & 0 \\
0 & 10+\cos(x_1)+\cos(x_2)
\end{array}
\right) , & B^{(6)}(x)=\left(
\begin{array}{cc}
1-({\rm e}^{{\mathbf i}x_1}+{\rm e}^{{\mathbf i}x_2})/2 & 0 \\
0 & 1
\end{array}
\right)
\end{array}
\end{equation*}
and define
\begin{align*}
f_{6}(x) &=Q(x)A^{(6)}(x)Q(x)^{T} \\
g_{6}(x) &=Q(x)B^{(6)}(x)Q(x)^{T}.
\end{align*}
The remark pointed out in the previous example about the outliers applies also in this case (compare Figures \ref{eigA6} and \ref{eigA6pre} with Figures \ref{caso6eig} and \ref{caso6eigpre}). In particular, we have found that the outliers bahave again like $O(\sqrt{\widehat n})$.
The singular values and the moduli of the eigenvalues of $T_n(f_6)$ are bounded, as shown in Figure \ref{caso6sv}. More precisely, a half of the eigenvalues of $T_n(f_6)$ is clustered at $[8,12]$, that is in the range of $A^{(6)}_{2,2}(x)=\lambda_2(f_6(x))$, the remaining part behaves as $A^{(6)}_{1,1}(x)=\lambda_1(f_6(x))$ (see Figure \ref{caso6eig}). Figures \ref{caso6svpre} and \ref{caso6eigpre} refer to the singular values and the eigenvalues of $T_n^{-1}(g_6)T_n(f_6)$. Let us observe that, although $0\in$ Coh$[{\cal ENR}(g_6)]$, the spectrum of $T_n^{-1}(g_6)T_n(f_6)$ is essentially determined by the spectrum of the function $g^{-1}_6f_6$. Table \ref{tab6} highlights once again that the GMRES with prenditioner $T_n(g_6)$ converges faster than its non-preconditioned version.
\end{Case}

In summary, we conclude that the proposed preconditioning approaches for non-Hermitian problems are numerically effective and confirm the theoretical findings. The encouraging numerical results show that there is room for improving the analysis and for providing a more complete theoretical picture, especially concerning the spectral localization and the number of outliers.

\end{document}